\documentclass{amsart}
 \usepackage{amsmath}
\usepackage{amssymb}
\usepackage{amsfonts}
      \makeatletter
     \def\section{\@startsection{section}{1}%
     \z@{.7\linespacing\@plus\linespacing}{.5\linespacing}%
     {\bfseries%\normalfont\scshape
     \centering
     }}
     \def\@secnumfont{\bfseries}
     \makeatother
  \usepackage{a4wide}
 
   \newtheorem{theorem}{Theorem}[section]
\newtheorem{lemma}[theorem]{Lemma}
\newtheorem{corollary}[theorem]{Corollary}

\theoremstyle{definition}

\newtheorem{remark}[theorem]{Remark}
\newtheorem{remarks}[theorem]{Remarks}

\numberwithin{equation}{section}

\def \a{{\alpha}}
\def \b{{\beta}}
\def \D{{\Delta}}

\def \d{{\delta}}
\def \e{{\varepsilon}}

\def \g{{\gamma}}

\def \l{{\lambda}}

\def \o{{\omega}}
\def \O{{\Omega}}
\def \p{{\varphi}}
\def \t{{\vartheta}}

\def \m{{\mu}}
\def \s{{\sigma}}

\def \A{{\mathcal A}}
\def \B{{\mathcal B}}

\def \E{{\bf E}\, }

\def \N{{\bf N}}

\def \P{{\bf P}}

\def \Q{{\bf C}}
\def \qq{{\qquad}}
 \def \R{{\bf R}}

\def \T{{\bf T}}

\def \Z{{\bf Z}}

\def \dd{{\rm d}}

\def \noi{{\noindent}}

\def\E{{\mathbb E \,}}

\def \T{{\mathbb T}}

\def\P{{\mathbb P}}

\def\R{{\mathbb R}}

\def\Z{{\mathbb Z}}
 
\def\Q{{\mathbb Q}}

\def\N{{\mathbb N}}

 \scrollmode
 
   \font\sevenrm= cmr10 at 7 pt
     \scrollmode
  
\def\ddate {\sevenrm \ifcase\month\or January\or
February\or March\or April\or May\or June\or July\or
August\or September\or October\or November\or December\fi\! {\the\day}, \!{\sevenrm\the\year}}

\scrollmode

  \title[Divergence Criteria in Ergodic Theory]{\ Criteria of Divergence Almost Everywhere in Ergodic Theory}
    
\begin{document}
    \author{Michel J.\,G.  WEBER}
  \address{IRMA, Universit\'e
Louis-Pasteur et C.N.R.S.,   7  rue Ren\'e Descartes, 67084
Strasbourg Cedex, France.
   E-mail:    {\tt  michel.weber@math.unistra.fr}}
   % \subjclass[2000]{ }
\footnote{  \par 
  \ddate{}}

\maketitle$$ \frac{}{\hskip 420 pt} $$
{\bf Abstract.}
%\begin{abstract}  
 In  this expository paper, 
we survey nowadays classical tools or criteria used in problems of convergence everywhere to build counterexamples: the Stein continuity principle, Bourgain's entropy criteria and Kakutani-Rochlin
 lemma, most  classical device  for these questions in ergodic theory. First, we state a $L^1$-version of the  continuity principle and give an example of its usefulness by applying it to some  famous problem on divergence almost everywhere of Fourier series.   Next we particularly  focus on    entropy criteria in $L^p$, $2\le p\le \infty$  and provide  detailed proofs. We also study the   link between the associated maximal operators  and the canonical  Gaussian process on $L^2$.  We further study the corresponding criterion   in $L^p$, $1<p<2$  using properties of $p$-stable processes. Finally we consider  Kakutani-Rochlin's
 lemma, one of the  most frequently used tool in ergodic theory, by   
     stating and proving a criterion for  a.e. divergence of weighted ergodic averages.
%\end{abstract}

$$ \frac{}{\hskip 420 pt} $$  This is the extended version of a paper that  appeared in  "Zapiski Nauchnyh Seminarov POMI", 
%(Gordin memorial volume)
 vol. 441,  
2015, ser. "Probability and Statistics 22", p.73--116.
% dedicated to the memory of M.I. Gordin  
%??????? ??????? ????????? ????, ??? 441
%"??????????? ? ??????????. 22"
%????????? ?.?. ???????, ?.?. ??????, ?.?. ??????
   
        %%%%%%%%%%%%%%%%%%%%%%%%%%%%%%%%%%%%%%%%%%%%%%%%%%%%%%%%%%%%%%%%%%%%%%%%%%%%%%%%%%%%%%%%%%%%%%%%%%%%%%%%%%%%%%%%%%%%%%  
%%%%%%%%%%%%%%%%%%%%%%%%%%%%%%%%%%%%%%%%%%%%%%%%%%%%%%%%%%%%%%%%%%%%%%%%%%%%%%%%%%%%%%%%%%%%%%%%%%%%%%%%%%%%%%%%%%%%%%  
%%%%%%%%%%%%%%%%%%%%%%%%%%%%%%%%%%%%%%%%%%%%%%%%%%%%%%%%%%%%%%%%%%%%%%%%%%%%%%%%%%%%%%%%%%%%%%%%%%%%%%%%%%%%%%%%%%%%%%  
%%%%%%%%%%%%%%%%%%%%%%%%%%%%%%%%%%%%%%%%%%%%%%%%%%%%%%%%%%%%%%%%%%%%%%%%%%%%%%%%%%%%%%%%%%%%%%%%%%%%%%%%%%%%%%%%%%%%%%  
\section{Introduction.}\label{s1}
 This is an expository paper on criteria of divergence almost everywhere in ergodic theory, and mainly Bourgain's entropy criteria in $L^p$, $2\le p\le \infty$. The paper is 
%aimed to be 
written in a self-contained  and  informative way: tools needed are presented, with (expected to be) helpful and sometimes historical comments, auxiliary results are included, as well as detailed and careful  proofs of main theorems. The preparation of this paper is thus   made in order to be also an efficient tool  for investigating these questions. This is in fact our main objective.
%care of details.
We do not study nor present applications. We refer for these to Bourgain \cite{B1}, \cite{B}, \cite{B2}. We also refer to   Rosenblatt and   Wierdl monograph \cite {RW}, to our monograph \cite{We1}   devoted to the study of these criteria and to Chapters 5 and 6 of our book  \cite{We} where applications of Stein continuity principle are also studied. We further refer to Lacey \cite{La}, Lesigne \cite{Le}, Berkes  and Weber \cite{BW} notably for other applications.     In writing the present paper, we   refered to  Chapter 6 of   \cite{We}. We were able to improve and simplify some  proofs and also complete it by new results. The entropy criterion in $L^p$, $1<p<2$ obtained in Weber \cite{We6} is stated and proved under a less restrictive commutation assumption, and we included the necessary material   from the theory of $\a$-stable processes (here $\a=p$) for the proof.  The metric entropy method     (first introduced by Strassen in the theory of Gaussian processes, see \cite{Du2}) is briefly and concisely presented for the need of the study. 

The paper is organized as follows. In Section \ref{s2}, we start with what is certainly, by the probabilistic argument used in its proof, the   basis of everything:
% concerning entropy criteria:
 the Stein continuity principle, the idea of which lies in Kolmogorov's seminal work  on harmonic conjugate functions and Fourier series \cite{Ko2}. A slightly forgotten aspect of this principle is that it is also a tool for producing counterexamples to almost everywhere questions. That point is developed in this Section. Next,   Section \ref{s3} is the central part of the paper and concerns Bourgain's entropy criteria and extensions of them. In Section \ref{s4}, we present auxiliary results concerning $L^p$-isometries, stable random variables and processes, variants of Banach principle, a metric comparison lemma and   basic Gaussian tools.  Section \ref{s5} is completely devoted to proofs of the results stated in Section \ref{s3}. We conclude the paper with  Kakutani-Rochlin
 lemma, one of the  most  classical devices   in ergodic theory. There are many applications of this result, also called Kakutani-Rochlin towers'
 lemma. We refer to Rosenblatt and   Wierdl monograph \cite {RW}. We illustrate it by stating and proving a criterion for  a.e. divergence of weighted ergodic averages,  based on  Deniel's construction \cite{D}.

%%%%%%%%%%%%%%%%%%%%%%%%%%%%%%%%%%%%%%%%%%%
%%%%%%%%%%%%%%%%%%%%%%%%%%%%%%%%%%%%%%%%%%%
%%%%%%%%%%%%%%%%%%%%%%%%%%%%%%%%%%%%%%%%%%%
%%%%%%%%%%%%%%%%%%%%%%%%%%%%%%%%%%%%%%%%%%%

\section{The Continuity Principle.}\label{s2}
Let $(X,{\A}, \mu)$ be a probability space  with a  $\mu$-complete $\s$-field $\A$.  
     Recall  some basic facts, and to begin,   recall that  the topology of   convergence in measure  on  
$L^0(\mu)$
 ($g_n\buildrel{\m}\over{\to} g$ if   $ \mu\left\{
|g_n-g|>\e\right\}\to 0$, for
any $\e>0$)   is metrizable and, endowed with the metric $d(f,g)=\int_X {|f-g|\over
1+|f-g|}\ d\mu $, $(L^0(\mu),d)$ is a
complete metric space.    
A mapping $V$ from a Banach space $B$ to ${L^0(\mu)}$   is said to be continuous in measure or $d$-continuous,
if for any sequence $(f,f_n,n\ge 1)\subset  B $, we have $  d(Sf_n, Sf)\to 0$
whenever $
  \|f_n-f\|\to 0$. 
\vskip 2 pt

Now let $1\le p\le \infty$ and consider sequence of  operators  $\{S_n,n\ge 1\}$,    $S_n\colon L^p(\mu)\to
L^0(\mu)$, which are continuous in measure.
By the Banach principle,  the set 
$$ \mathcal F(S)= \big\{f\in L^p(\m): \{S_nf, n\ge
1\}  \ \text{converges}\  \m-\text{almost everywhere}\big\}$$
  is closed in $L^p(\m)$    if
and only if:
   
{\it There exists a non-increasing function $C\colon \R^+\to \R^+$
with $\lim_{\a\to \infty} C(\a)=0$, and such that for any $\a\ge 0$ and any  $f\in L^p(\m)$,} 
  $$ \m \big\{S^*f>\a \|f\|_{p, \m}\big\} \le C(\a) \qq \quad \text{where} \ \ S^*f= \sup_{n\ge 1} |S_nf| .$$  
 \vskip 2 pt \noi 
When
the sequence $S$ commutes with a  sequence $  \{ \tau_j, j\ge 1\}$ of measurable
transformations of  $X$ preserving $\m$  and mixing in the following
sense:
$$
\forall A,B\in \mathcal{ A}, \ \forall \a>1,\   \exists j\ge 1\colon\quad \m(A\cap
\tau_j^{-1}B)\le \a \m(A)\m(B),\eqno(H)
$$
and $1\le p\le 2$, then by the continuity
principle $C(\a) = \mathcal{ O}(\a^{-p})$. 

\vskip 1 pt \noi This is fulfilled  if
$S$ commutes with an ergodic endomorphism of $(X,\mathcal{ A}, \mu)$. So
that the study of the convergence almost everywhere of the sequence
$S$  amounts,  under appropriate  commutation assumptions,   to
establish  a maximal inequality  and to exhibit  a dense subset of
$L^p(\m)$  for which the convergence almost everywhere already
holds.    

Before stating the Continuity Principle,    \begin{theorem} 
Suppose that  $\{S_n,n\ge 1\}$  is  a sequence of  operators, $S_n\colon L^p(\mu)\to
L^0(\mu)$, $1\le p\le 2$, which are continuous in measure and
satisfy  the commutation assumption  $(  H)$.
  Then the  following  properties are equivalent:
 \begin{eqnarray*} {\rm (i)} & &   \forall f\in L^p(\mu),\quad  \mu\{x: S^*f
(x)<\infty\}=1.
\cr {\rm (ii)} & & \exists
0<C<\infty  : \forall f\in L^p(\mu), \quad   \sup_{\lambda  \ge
0}\lambda  ^p\mu\{x:S^*f(x) >\lambda  \}\le
 C\int_X |f|^p\ d\mu.
 \end {eqnarray*}
\end{theorem}
\begin{remark}\label{sawyer} If $p>2$, the same conclusion holds for positive operators ($S_nf\ge 0$, if $f\ge 0$). This was proved later by Sawyer in \cite{Sa}.
 \end{remark}
 The proof   combines quite subtely and   remarkably,
analysis  and probability. The commutation property of the operators $S_n$ is crucial, and makes the proof possible.  Earlier, Kolmogorov used already in \cite{Ko2} the fact that the operators
$$ H_nf(x) = \int_{|t|>1/n} f(x-t) \frac{\dd t}{t}, \qq \quad f\in L^1_{{\rm loc} }(\R)$$
all commute with translations to prove the similar   inequality: let $H^*f(x)= \sup\{ |H_nf(x)|, n\ge 1\}$, then 
$$  \sup_{\lambda  \ge
0}\lambda m\{ x:  H^*f(x) > \l\}\le  C\int_{\R} |f( x) |{\dd x} ,$$ 
$m$ denoting here the Lebesgue measure on $\R$. The setting considered in \cite{S} is group theoretic:   $\O$ is a commutative compact group,   $\m$ is the Haar measure and $S_n$ are commuting with translations.   
Sawyer \cite{Sa} showed that this setting is not necessary and that a general principle can be derived under  the above assumptions. We refer to the 
nice  monograph of Garsia \cite{G}. 
\vskip 5 pt  
 The Continuity Principle is {\it not only} a tool for studying  integrability of maximal operators $ S^*f$,
but also a device for producing counterexamples in problems of convergence almost everywhere. This was observed and studied by Stein \cite{S}, but also by
Burkholder \cite{Bu} and Sawyer \cite{Sa}. That important aspect of this principle seems to have been forgotten over the years. In \cite{S}, Stein  
has established other forms of this  principle with quite striking applications,     proving notably negative convergence results.   One
of these applications concerns a deep result of Kolmogorov
 \cite{Ko3},  \cite{Ko4} showing  the existence of an integrable function  whose Fourier series diverges almost everywhere. The proof is known to be very difficult.  Using a
suitable form of his principle for the space $L^1(\m)$, Stein could  refine  and also provide  a simpler proof of  Kolmogorov's result. 
Convergence criteria   for this space are not frequent, and reveal   crucial in many deep questions. We  recall it now. 
\vskip 2 pt  
We assume here that $X$ is a commutative compact group     and denote by ``$+$'' the
group operation. Let $\m$ be the unique invariant measure, the Haar
measure on $X$.
 Let  $\mathcal{ C}(X)$ 
be  the space of continuous functions on $X$, with the
supremum norm, and $\mathcal B(X)$ be  the space of finite Borel
measures on $X$ with the usual norm. Let $\{S_n, n\ge 1\}$ be a
sequence of operators. We assume:
\smallskip

 (a)\ Each $S_n$ is a bounded operator from $L^1(\m)$ to $\mathcal{ C}(X)$.\smallskip

(b)\ Each $S_n$ commutes with translations.
\smallskip

   By   Riesz's representation of bounded linear functionals on
$L^1(\m)$,   conditions (a) and (b) are
equivalent with
\smallskip

 (c) $S_n f(x) = \int_X K_n(x-y) f(y) \m(dy)$, where $K\in L^\infty (X)$.

\smallskip
\noi Such an operator has a natural extension to a bounded operator
from $\mathcal B(X)$ to $L^\infty (\m)$, which we   again denote by $S_n$. Notice
that this extension still commutes with translations. Similarly, we
also write $S^*\nu= \sup_{n\in \N}|S_n\nu| $.

\begin{theorem}\label{sl1}
 Under   assumptions (a) and (b), the  following
assertions are equivalent:
 \begin{equation}\label{sl11} \forall f\in L^1(\m),\quad  \m\{x: S^*f
(x)<\infty\}=1, 
 \end{equation} 
\begin{equation}\label{sl12}
 \exists 0<C<\infty \colon\ \forall \nu\in \mathcal B(X),\
 \quad  \sup_{\lambda  \ge 0}\lambda\,  \m \Big\{x:S^*\nu(x) >\lambda \int_X |d\nu|   \Big\} \le   C  . 
 \end{equation}
\end{theorem}
 To give an idea of its strength, let us show how recover   Kolmogorov's theorem. Introduce the necessary notation. We denote 
 throughout this article   by $\T$ the circle $\R/\Z\sim[0,1[$. 
 
 Take $X=\T$ and  $\m$ be  the normalized Lebesgue measure  on $\T$.
   Let $S_n(f)$
denote here the partial sum of order $n$ of the Fourier series of $f$,
and more generally let  $S_n(\nu)$    be the partial sum  of order $n$ of the
Fourier--Stieltjes expansion of a Borel measure $\nu$. Recall that  for any  integrable $f$,
 $$S_nf(x) -S_mf(x) =\mathcal{ O}  ( \log|m-n|  ), \qq m,n \rightarrow \infty, $$
almost everywhere. Stein proved  the
following refinement: 
\begin{theorem}\label{ks} Let $\p(n)$ be any function tending  to zero as $n$ tends
to infinity. Then  there exists an integrable function $f(x)$ such that the
more restrictive property
\begin{equation}\label{srefin}S_n(f )(x) -S_m(f)(x)  =\mathcal{ O}  ( \p(|m-n|) \log|m-n|  )   
\end{equation}
is false for almost every $x$.
\end{theorem}  This of course implies Kolmogorov's theorem. For the proof, consider the family of operators
$$\D_{(m,n)}f= {S_n(f)  -S_m(f) \over \p(|m-n|) \log|m-n|}. $$
These operators satisfy  conditions (a) and (b) of Theorem \ref{sl1}. A lemma is necessary.

\begin{lemma}\label{ls1}
There exists an absolute constant $C$ such that for any integer $k$, there exists a measure $\nu$ on $\mathbb{T}$ with $\int_\mathbb{T} |d\nu| =1$ and
$$
\sup_{n,m  : |n-m|=k}\big|S_n(\nu)-S_m(\nu)\big| \ge C \log k\qq \text{ almost  surely}.
$$
\end{lemma}

\begin{proof} Let $x_1,\dots, x_N$ be some points of $\mathbb{T}$ to
be specified later, and set
 $
\nu ={1\over N}\sum_{i=1}^N \d_{x_i}$,
 where $\d_x$ denotes the Dirac measure at point $x$. Then $\int_\mathbb{T} |d\nu| =1$. Plainly,
$$
S_n(\nu)(x) - S_m(\nu)(x)= {2\over N}\sum_{i=1}^N {\cos \pi (n+m+1) (x-x_i)
\sin\pi (n-m)(x-x_i)\over \sin\pi (x-x_j)}.
$$
Write $k=n-m$, $\ell = n+m+1$. Assume that $k$ is odd. Then $\ell$
must be even, but this is the only restriction on $\ell$.  We choose
the $x_i$ to be linearly independent over $\Q$, and such that they
are very close to $i/N$. It is easily seen then,  that for almost
every $x$, the $x-x_i$ are linearly independent over $\Q$. Choosing
$\ell$ large enough, depending on $x$, we have
$$\sup_{n,m  : |n-m|=k} |S_n(\nu)(x)-S_m(\nu)(x) |= {2\over N}\sum_{i=1}^N { |\sin\pi k(x-x_i)|\over |\sin\pi (x-x_j)|} .$$
The fact  that  $x_i$ are  very close to $i/N$ and   $N$ is large
enough, shows that the sum on the right is close to its
 integral counterpart, and so   exceeds
  half of its value. Therefore,
$$\sup_{n,m  : |n-m|=k} |S_n(\nu)(x)-S_m(\nu)(x) |\ge {1\over 2}  \int_\mathbb{T} { |\sin\pi k(x-y)|\over |\sin\pi (x-y)|}dy\ge C\log k, $$
as required.\end{proof}
Now we prove Theorem \ref{ks}. Suppose on the contrary that   property
\eqref{srefin} were true with positive probability, and this for any $f\in L^1(\mathbb{T})$. Let   $\tau$ be an irrational rotation of $\T$, thereby an   ergodic measure preserving transformation.
Note that if  
 $ A= \big\{ \sup_{n,m} |D_{(m,n)}f|<\infty\big\}$, then $\tau ^{-1}(A)\subset A$.  
By Birkhoff's theorem, this suffices to imply   that $\m(A)=1$.  So that  the operators
$\D_{(m,n)}f$ would satisfy   condition \eqref{sl11}.   
Consequently,   the maximal operator
$$
\nu \mapsto \D^*(\nu) :=\sup_{n,m }\Big|{S_n(\nu)(x)-S_m(\nu)(x)\over \p(|m-n|) \log|m-n|}\Big|
$$
would satisfy \eqref{sl12}. Therefore this would imply the existence of a constant $C_0$
such that for any $\nu\in  \mathcal B(M)$  with $\int_\mathbb{T} |d\nu| =1$, and
any $t  \ge 0$,
  $
  t\m \{x : \D^*\nu(x) >t  \}\le   C_0  $.  
\vskip 1 pt
 Let $k$ be a positive integer, which we choose sufficiently large to
ensure that  $\log k> (2C_0)/C $, where $C$ is the same constant as
in Lemma \ref{ls1}. Apply this  for $t =C(\log k)/2$; then,
$$
\m\big\{ x : \D^*\nu(x) >{C\over 2}\log k \big\}\le   {2C_0 \over  C\log k} <1.
$$
By Lemma \ref{ls1}, there exists $\nu\in \mathcal B(M)$
with $\int_\mathbb{T} |d\nu| =1$ such that $\D^*\nu\ge C\log k$ almost
surely. Hence a contradiction and condition \eqref{sl11} cannot hold.
 Therefore there exists
an integrable function such that property \eqref{srefin} is false for
almost every $x$.

For recent   results related to Kolmogorov's theorem, see Lacey's very nice paper \cite{La1}, Section 9.3. We   refer to \cite{S} (see also \cite[Chapter 5]{We}) for several other applications of this kind. 
\vskip 5 pt
To  
$f\in L^2(\m)
$,   associate the sequence  in which we set $T_j f= f\circ \tau_j$, 
   \begin{equation} \label{fj0} F_{J,f} =\frac{1}{\sqrt J}\sum_{1\le j\le J} g_j T_jf,\qq (J\ge 1)
   \end{equation}
where $g_1,g_2,\ldots $ are i.i.d. standard Gaussian random variables, defined on a common joint probability space $(\O, \B, \P)$. 
  
  These random elements (with Rademacher weights instead of Gaussian's) are    key tools in Stein's proof. 
  The same elements (sometimes with stable weights) are
also playing a central role in   Bourgain's entropy criteria and  extensions obtained by the author. The notation used 
in \eqref{fj0} will be later formalized to include these cases, see \eqref{fjgen}. 
Lifshits and Weber studied in \cite{LW1}, \cite{LW2} and \cite{We3} their oscillations properties and the tightness properties of their laws.
\vskip 2 pt
The Continuity Principle is  established in an indirect way in \cite{S}.  A direct proof   with Gaussian weights (as in the proofs of Bourgain's entropy criteria) was given in \cite{We}. 
\vskip 2 pt
 We close this section with an interesting and somehow intriguing observation. The key point of
the proof [at this stage, the Banach Principle is not yet applied] is contained in the following inequality (see  \cite[p.\,211-212]{We})
   \begin{equation}
{ n\mu\{S^{*}(f) >M(1+n)^{1/ p}\}-2 \over n
\mu\left\{S^{*}(f) >M(1+n)^{{1/ p}}\right\}}
\le 8\,
\E
   \mu\big\{ S^{*}(F_{n,f}) >c M\big\}  , \end{equation}
which holds for any  $M>0$,   any integer $n\ge 2$, and $c$ is a numerical constant.  
 Now  by simply  permuting the order of integration, we get
  \begin{equation}
{n\mu \{S^{*}(f) >M(1+n)^{{1/ p}} \}-2 \over n\mu \{S^{*}(f) >M(1+n)^{{1/ p}} \}}\le 8\int_{X}
\P \big\{S^{*}(F_{n,f})>   c M  \big\} \ d\mu, 
\end{equation}
where this time, $S^{*}(f)$ is controlled by its random counterpart of $S^{*}(F_{n,f})$
for an appropriate choice of  the integer $n$. Therefore a good control of the random counterpart also provides a good control of the initial sequence.
\vskip 2 pt  {\bf Notation.} We reserve the letter \,$g$\,  to denote throughout  an $\mathcal{ N}(0,1)$ distributed
random variable. An index or a sub-index always denotes an infinite increasing sequence of positive integers.  
\vskip 2 pt

%%%%%%%%%%%%%%%%%%%%%%%%%%%%%%%%%%%%%%%%%%%
%%%%%%%%%%%%%%%%%%%%%%%%%%%%%%%%%%%%%%%%%%%
%%%%%%%%%%%%%%%%%%%%%%%%%%%%%%%%%%%%%%%%%%%
%%%%%%%%%%%%%%%%%%%%%%%%%%%%%%%%%%%%%%%%%%%

\section{Metric Entropy Criteria.}  \label{s3} Throughout the remainding part paper, let    $S$ denotes, unless explicitly   mentioned,  a sequence of continuous operators $ S_n\colon L^2(\mu)\to
L^2(\mu)$, $n\ge 1$.
 Using the theory of Gaussian processes, Bourgain has established in  \cite{B}     two   very useful
criteria linking   the regularity properties (boundedness, convergence almost everywhere) of the sequence $S$
  with the metric entropy properties  of the
sets $C_f
$ below. 
\vskip 2 pt   The concept of entropy numbers (namely covering numbers) associated with  a metric space  is old; it was invented by
 Kolmogorov as a device for classifying functional spaces. See Kolmogorov \cite{Ko2}, Kolmogorov and  Tikhomirov \cite{KT},  Lorentz \cite{Lo}. In many situations, these numbers are
computable  (typical examples of sets are ellipsoids, see \cite{Du1}); hence their interest.      
 Recall that any compact set
in a separable Hilbert space is included in some ellipsoid, see  Raimi \cite{R} and for relations between their entropy numbers, see Helemski\u \i\,  and Henkin \cite{HH}.
\vskip 2 pt  Bourgain also showed, by means of imaginative constructions,   how to apply   these criteria to several analysis problems,
among them Marstrand's disproof of Khintchin's   Conjecture,   a problem posed   by Bellow and a question raised by Erd\"os.  This is a quite striking achievement, which adds a new chapter
to Stein's Continuity Principle.  We believe that Bourgain's approach goes beyond the setting explored in \cite{B1}, \cite{B}, \cite{B2} and 
 should   deserve   further investigations. 
 The author has obtained in \cite{We6}, \cite{BW}, \cite{We1}
   extensions of these criteria and applied them to similar
questions. He further   studied in \cite{We2}, \cite{We4}, \cite{We5} the geometry of the sets $C_f$ defined in \eqref{cf}, as well as  and their natural
extension
$C(A)= 
 \left\{S_n(f),n\ge 1, f\in A\right\}$, in which  $A$ is an arbitrary subset of $L^2(\m)$. See Appendix. We also refer    to Talagrand \cite{T1} where
this 
   question  was investigated  in a larger context.
 \vskip 2 pt
Introduce the following commutation condition:
\vskip 2 pt   {\rm (C)}  {\it There exists a sequence $ \{T_j,\ j\ge
1 \}$  of positive isometries of $L^1(\mu)$, with $T_j1=1$,
  such that  for any $f\in L^1(\mu)$, 
\begin{equation}\label{Ctj}   \lim_{J\rightarrow \infty}\big \|\,{1\over J}\sum_{j\le
J}T_jf -\int f d\mu\,\big\|_{1,\mu}=0  , 
\end{equation}
and commuting  with $S$, $
S_n(T_jf)= T_j(S_nf)$ for any  $f\in L^2(\m)$.} 
\vskip 3 pt 
 Consider    for   $2\le
p\le\infty$,  the following convergence   property 
$$ \mu\big\{  \{S_n(f),n\ge 1 \} \text{ converges}
\big\}=1,\qq \qq \text{for all $f\in L^p(\mu)$.}  \leqno(\mathcal{
C}_p)$$  
Set
\begin{equation}\label{cf} C_f=\left\{S_n(f),n\ge 1\right\}, \qq \qq    f\in L^2(\mu). 
\end{equation}
  
  Bourgain's first criterion  \cite[Proposition 1]{B} shows that if $(\mathcal{
C}_p)$   holds for some $2\le p<\infty$, the sets $C_f$ cannot be too large. More precisely,  
 
%%%%%

\begin{theorem} \label{t1}
  Let $S$ be a sequence of $L^2(\mu)$ contractions  satisfying   condition {\rm (C)}. 
   Assume    that $(\mathcal{C}_p)$ holds  for some  $2\le p<\infty$. 
  Then  there exists a    numerical constant $C_0$ such that for any $f\in L^p(\mu)$,
\begin{eqnarray*} \sup_{\e >0} \varepsilon\sqrt{  \log
N_f(\varepsilon)}
 &\le & C_0\|f\|_2, 
\end{eqnarray*}  
where   for any $\e>0$,  $N_f(\varepsilon)$ denotes the minimal   number
of $L^2(\mu)$ open balls of radius~$\e$, centered in $C_f$  and
enough to cover $C_f$.\end{theorem}
 \begin{remark}\label{r1}  By using covering properties of ellipsoids, one can show that the above entropy estimate  is optimal for convolutions on the
circle; and thus admits no improvment. See \cite[p.\,47]{We1}. However, it can be far from optimal on typical examples. Let $S_nf =\frac{1}{n} \sum_{ j\le
n } T^j f$, where $T$ is some measure preserving transformation on $(X, \A, \m)$.  By a theorem of Talagrand
$N_f(\e)\le C\max( 1, \|f\|_{2,\m}^2/\e^2)$, $0<\e \le \|f\|_{2,\m}$ where $C$ is an absolute constant. See \cite{T1}, \cite[Theorem 1.4.1]{We}.
   \end{remark}
Bourgain's second criterion  \cite[Proposition 2]{B} states as follows.
\begin{theorem} \label{t2}
Let $S$ be a sequence of $L^2(\mu)$ contractions  satisfying   condition {\rm (C)}. 
    Assume  that   $(\mathcal{
C}_\infty )$ is fulfilled.
  Then for any real $\d>0$,
$$  C(\delta)=\sup_{f\in L^\infty(\mu),\,  \|f\|_{2, \m}\le 1}N_f(\delta)<\infty.
 $$
\end{theorem}

  A starting point of the proof  is a version (see \cite[(9)]{B}) of the Banach principle for $L^\infty(\m)$, namely the fact that the convergence property $(\mathcal{
C}_\infty )$ implies that
\begin{equation}\label{0cont} \sup_{\|f\|_{\infty, \m}\le 1, \|f\|_{2, \m}\le \e}\int_X \frac{S^*f}{1+S^*f} \, \dd \m \ \to \ 0, \qq \quad {\rm as} \quad \e
\to 0.
\end{equation}   
  This result was   established few after by Bellow and Jones in \cite{BJ}. The proof is however lenghty and indirect. It is possible to provide a direct and short proof,  similar
to the one of the standard Banach principle,  see  \cite[Theorem 5.1.5]{We}. 

\vskip 2 pt   Note that the integrability of  
$S^*f$, which is required in  the proof of Theorem \ref{t2}   in \cite{B},  is not ensured by the assumption  made in Theorem \ref{t2}.
 This is for instance guaranteed
when $S_n$ are $L^2(\m)$-$L^\infty(\m)$ contractions, which is the case of all applications given in \cite{B}. Moreover, Bourgain's proof runs with no modification using \eqref{0cont} at the conclusion.
 
 \vskip 2 pt
 
 Given a separable Hilbert space $H$, recall that   the
 canonical   Gaussian (also called isonormal) process $Z=\{Z_h,h\in H\}$ on $H $  is the centered Gaussian   process  with covariance function
 $$
 \Gamma(h,h')=\langle h,h'\rangle,  \qq \qq  h,h'\in H.
 $$
 Let $\{h_n, n\ge 1\}$ be a countable orthonormal basis  of $H$.   Let also $ 
\{g_n,n\ge 1\}$ be a    sequence
  of i.i.d.\  $\mathcal{ N}(0,1)$
distributed random variables on a basic probability space $(\O,
\A,\P)$. Then $Z$ can be defined as follows: for any $h\in H$,
$$
Z_h = \sum_{n=1}^\infty g_n \langle h,h_n\rangle.
 $$
  A subset $A$  of $H$ is a GB set (for Gaussian bounded)  if
the restriction of $Z$ on $A$ possesses a version which is sample
bounded.  Further,  $A$ is a GC set  (for Gaussian continuous)  if the restriction of $Z$ on $A$ possesses a version which is sample $\|\cdot\|$-continuous. These notions were   introduced in Dudley \cite{Du1}.

A countable  subset $A$ of $H$  is a  GB set   if $\E  \sup_{h\in A}\left|Z(h)\right| <\infty$, or   equivalently $\E  \sup_{h\in A} Z(h)  <\infty$, since as is well-known,
  $$
\E   \sup_{h\in A} Z(h)  \le  \E  \sup_{h\in A}\left|Z(h)\right| \le
2\E     \sup_{h\in A} Z(h)  + \inf_{h_0\in A}\E  |Z(h_0)| .
 $$ 

Under assumptions of Theorem \ref{t1},  Bourgain has also  shown    that the sets $C_f$ are  GB sets. Some remarks are in order. It is
not necessary to assume  that $S_n$ are $L^2(\m)$-contractions. Moreover the conclusion   remains true under a weaker condition than
$(\mathcal{ C}_p)$.    Theorem \ref{t1} can be reformulated as follows.
  \begin{theorem} \label{t3}
   Let $S_n\colon L^2(\mu)\to L^2(\mu)$, $n\ge 1$  be  continuous operators  satisfying assumption  {\rm (C)}.
     Assume  that for some $2\le p<\infty$,  $$
 \m \big\{  \sup_{n\ge 1}|S_n(f)| <\infty \big\}=1,  \qq \qq \text{for all $f\in L^p(\mu)$}   .   \leqno(\mathcal{ B}_p)
 $$
 Then  for any $f\in L^2(\mu)$,  the sets  $C_f$  are   GB sets   of $L^2(\mu)$. 
Further there exists a  numerical constant
$C_1$ 
 and a  constant $C_2$ 
 such that for any $f\in L^2(\mu)$,
  \begin{eqnarray*} C_1\sup_{\e >0} \varepsilon\sqrt{  \log
N_f(\varepsilon)}
  \le  \E \  \sup_{n\ge 1}Z(S_n(f))
\le C_2\|f\|_{2, \m}. 
\end{eqnarray*}  
\end{theorem}
 
  The use of   
   the   fact
  that if  $N(X)$ is  a Gaussian semi-norm, then  \begin{eqnarray}\label{fer} \P\{N(X)\le s\}>0 & & \Rightarrow 
   \qq \E N(X)\le \frac{4s}{\P\{N(X)\le s\}},
\end{eqnarray} 
slightly simplifies the proof, which otherwise is very similar (\cite{We}). Estimate \eqref{fer} will be frquently used in the sequel.
\begin{remark}\label{GC} One can naturally wonder whether property $(\mathcal{ C}_p)$ analogously implies that the sets $C_f$ are GC sets. 
This question was investigated in \cite[\S\,5.2.2]{We1}, where in Theorem 5.2.4 it is shown that the answer is positive when $X=\T$ and $S_n$ are commuting
with rotations.  
\end{remark}   
   \vskip 7 pt

 Note before continuing that  when $\int_X   S^*f \, \dd \m$ is finite,  no explicit  link  with
  $$\E  \sup_{n\ge 1}Z(S_n(f))$$ 
 can be drawn from Theorem \ref{t3}. In Theorem \ref{t5} below, this is established. A general  inequality  valid   for arbitrary partial maxima,  can be directly indeed derived from  condition {\rm (C)}  only.  Before, we add further comments. First, say a few words on the way the commutation condition $(C  )$  links 
   $Z$ and $S$. This  explains easily. 
Let $f\in L^2(\mu)$ and let  $I$ be a  finite  set of integers.  Then one  derives from $(C  )$, that there exists  an  index
$\mathcal{ J}$  such that the   two-sided inequalities
$$
    \frac{1}{2}  \| S_n(f)-S_m(f) \|_{2,\mu}\le    \|  S_n(F_{J,f})-S_m(F_{J,f})(x) \|_{2,\P}
 \le 2  \| S_n(f)-S_m(f) \|_{2,\mu} ,  
 $$ 
hold  true for all  $n,m\in I$ and all $J\in \mathcal{ J}$,   and for      all $x$ in a measurable set of positive measure.
   See Lemma \ref{a2}.     Theorem \ref{t1} is  obtained as a straightforward application of the Banach principle, and Slepian's inequality combined with
Sudakov's minoration (Lemma
\ref{a4}).  
    \vskip 3 pt
  Bourgain essentially applied Theorem \ref{t2}, and this in the case  $X=\T$, and $T_j$ are translation or dilation operators. In either case, condition
\eqref{Ctj} is obviously satisfied. The counter-examples are built on functions of the type  
$$ f=\frac{1}{\sqrt{\#(F)}}\sum_{ n\in F} e_n,\qq\qq\qquad (e_n(x) =e^{2i\pi nx})$$ 
where $F$ are   specific   arithmetic sets. These elements, as well as all $T_jf$, $j\ge 1$, not only belong to $L^p(\m)$ but  also  to many more specific spaces. So that for Banach spaces
 $B$ such that $B\subset L^2(\m)$, a requirement on $f\in B$ like 
\begin{equation}\label{tjr}  T_jf \in B, \qq\qq   \forall j\ge 1,
\end{equation}
is frequently non void.  Call $\mathcal R(B)$ the set of these elements.  Then $F_{J,f}\in  B$ whenever $f\in \mathcal R(B)$. If  $B=L^p(\m)$ for instance, then by Corollary \ref{c0} and Lemma \ref{a1}, $\mathcal R(B)=B$.  
\begin{theorem}\label{t5}
 Let $S$ be
satisfying  condition {\rm (C)}.   Let  additionally
$I$ be a  finite set of integers and  $0<\e<1$. 
 \vskip 2 pt \noi {\rm (i)}     
Let   $B\subset L^2(\m)$ be a Banach space with norm $\|.\|_{B}$. Let  $f\in \mathcal R(B)$. Then there exists a partial index
$\mathcal J$ such that  for any $J\in \mathcal J$, any positive increasing convex function $G\colon {\mathbb{R}}^+\to {\mathbb{R}}^+$ 
 $$
 \sqrt{1-\e}\  \E \sup_{n\in I} Z(S_n(f)) \le    \E    \|F_{J,f}\|_{B}  \sup_{
   \|h\|_{B}\le 1  }\int
\sup_{n\in I} |S_n(h)| \ d\mu .$$
And 
\begin{align*} \qq     \E    G\Big( \sqrt{1-\e}  \sup_{n,m\in I}
\big|Z(S_n(f))-Z(S_m(f))\big|\Big)\qq\qq\qq\qq \qq\qq\qq\qq 
 \cr      \le     \E    \|F_{J,f}\|_{B}  \sup_{
   \|h\|_{B}\le 1  }\  \E \int_{X} G\big(\sup_{n,m\in I} \big|(S_n-S_m)(h) \big|\big)
d\mu.   \end{align*}

\vskip 2 pt \noi {\rm (ii)} In particular, for any $f\in L^p(\m)$ with $2\le p<\infty$,
 \begin{eqnarray*}
    \sup_{\|f\|_{2,\mu}\le 1}\E \sup_{n\in I} Z(S_n(f)) &\le  &  C_p     \sup_{
   \|h\|_{p,\mu}\le 1  }\int
\sup_{n\in I} |S_n(h)| \ d\mu,  
\end{eqnarray*} 
where $C_p=\|g\|_p/\|g\|_2$, recalling the notation used.
Further 
 \begin{align*} \qq  \sup_{\|f\|_{2,\mu}\le 1}   \E    G\Big(   \sup_{n,m\in I}
\big|Z(S_n(f))-Z(S_m(f))\big|\Big)\qq\qq\qq\qq \qq\qq 
 \cr      \le    C_p \sup_{
   \|h\|_{p,\m}\le 1  }\  \E \int_{X} G\big(\sup_{n,m\in I} \big|(S_n-S_m)(h) \big|\big)
d\mu.   \end{align*}
\end{theorem}
 \vskip 3 pt  
We have the following criterion providing a general  form of Theorem
\ref{t3}.
\begin{theorem}\label{gen}
Let $S$ be satisfying  assumption {\rm (C)}.  Let $B\subset L^2(\m)$ be a Banach space with norm $\|.\|_{B}$. Assume that the
following property is fulfilled:
$$
\m \big\{  \sup_{n\ge 1}|S_n(f)| <\infty \big\}=1,  \qq  \quad \forall f\in B .
$$
Then  there exists a constant $K$ depending on $S$ and $B$ only such that
 \begin{eqnarray*} \E \sup_{n \ge 1} Z(S_n(f)) &\le &  K \,\limsup _{H\to \infty}  \E \| F_{H,f}\|_{B} ,   \qq  \quad \forall f\in \mathcal R(B) .  
\end{eqnarray*} 
 \end{theorem}   

Let us derive a criterion  which has been recently applied in \cite{BW} to show the optimality of a famous theorem of Koskma.  Let $\{h_n, n\in \Z\}$ be a
countable orthonormal basis  of
$L^2(\m)$ and use the notation 
$f\sim
\sum_{n\in
\Z}  a_n(f)h_n
$, 
$
\sum_{ n\in \Z}   a^2_n(f)<\infty   $,  if $f\in L^2(\m)$.  
    Given    a sequence of positive reals $w=\{w_n, n\in \Z\}$ with  $w_n \ge 1$,  we recall that $L^2_w (\m)
$ is the sub-space of  $L^2 (\m) $ consisting of functions  such that
 $$ \sum_{ n\in \Z} w_n a^2_n(f)<\infty.  $$
This is  a  Hilbert space with scalar product defined by  $\langle
f,h\rangle= \sum_{ n\in \Z} w_n a_n(f)a_n(h)$, 
 and norm 
$$\|f\|_{2,w} =\big(\sum_{ n\in \Z} w_n a^2_n(f)\big)^{1/2}  .$$
The space $L^2(\m)$ corresponds to the case $w_n\equiv 1$.  And  $L^2_w(\m) $ trivially contains any $f$ such that $a_n(f)=0$ except for finitely many $n$.

\begin{corollary}\label{t4}
Let $S$ be satisfying  assumption {\rm (C)}.  Assume that the
following property is fulfilled:
$$
\m \big\{  \sup_{n\ge 1}|S_n(f)| <\infty \big\}=1,  \qq \qq \text{for all $f\in   L^2_w(\m) $} .
$$
Then  there exists a constant $K$ depending on $S$ and $w$ only such that

$$ \sup_{\e >0} \varepsilon\sqrt{  \log
N_f(\varepsilon)}   \le K \limsup_{J\to \infty} \E \|F_{J, f
}\|_{2,w} ,   \qq \qq \text{for all $f\in   \mathcal R( L^2_w(\m )) $}    .$$
 \end{corollary}
\begin{remark} Let $X=\T$, $\m$ the normalized Lebesgue measure and let $T_j$ be dilation operators, $T_jf(x)=f(jx)$. Then any finite trigonometric sum belongs to $ \mathcal R( L^2_w(\m ))$.
\end{remark} 
  We refer to \cite[Chapter 6]{We} for a   study  of the link between the partial maximum operators ($I$
being a   set integers).  
\begin{equation}\label{0cont1}   \sup_{\|h\|_{\infty, \m}\le 1\atop  \|h\|_{2, \m}\le \e}\int_X \sup_{n\in I} |S_n(h)|\dd \m 
\qq {\rm and} \qq  \sup_{    \|f\|_{2,\m}\le 1}\E  \sup_{n\in I }Z(S_n(f)).
 \end{equation}  
In the theorem below, we provide a quantitative link.
 \begin{theorem}\label{t7}Let $S_n$, $n\ge 1$,  be  $L^2(\m)$-$L^\infty(\m)$ continuous operators verifying condition ${\rm (C )}$. Let  $I$ be any set of integers with cardinality $M$. For any reals $A>0$, $R>0$, it is true that
 \begin{eqnarray*}
 \sup_{ \|f\|_{2,\mu}\le 1}  \E  \sup_{n\in I} Z(S_n(f)) &\le& 6\, 
\sqrt{  M}S_1(I)\,\exp\{- A^2 / 8\}+ A (
\sqrt 2  )S_2(I)\,   e^{- {R^2}/{4 } }    \cr &  & + A
\sup_{\|h\|_{\infty,\mu}\le 1\atop \|h\|_{2,\mu}\le R/A}\int_{X }  \sup_{n\in I}|S_n(h)| \ d\mu   ,
\end{eqnarray*}
where  $ S_1(I) = \max_{n\in I}  \|S_n\|_2 $, $S_2(I) =
  \max_{n\in I} \|S_n\|_\infty  $, and $$\|S_n\|_2= \sup_{\|f\|_2\le 1}  \|S_n(f)\|_2 \qq  \|S_n\|_\infty= \sup_{\|f\|_\infty \le 1}\|S_n(f)\|_\infty .$$
 \end{theorem}  
\begin{remark} It is not complicate to derive from this bound Theorem \ref{t2}, for  $L^2(\m)$-$L^\infty(\m)$ contractions.

\end{remark} \vskip 4 pt Now consider the spaces $L^p(\m)$, $1<p< 2$.   A corresponding  entropy criterion can be also established. 
\begin{theorem}\label{t6}Let
$1<p\le  2$ with conjugate number $q$.  
Consider a sequence $S=\left\{S_n,n\ge 1\right\}$ of continuous
operators from    $L^p(\mu)$ to $L^p(\mu)$. Assume that condition ${\rm (C )}$  is satisfied.

 Further assume that for some real $0<r<p$, property  $(\mathcal{ B}_r)$ is satisfied. Then there exists a constant $C(r,p)<\infty$ depending
on $r$ and $p$ only, such that for any $f\in L^p(\mu)$,
$$  \sup_{\e >0} \e \big( \log  N_f^p(\e )\big)^{1/ q}
 \le C(r,p)\|f\|_{p, \m} ,
 $$ where $N_f^p(\e )$ is the minimal number of  open
$L^p$-balls of radius $\e  $, centered in $C_f$ and enough to cover
it. Further $ C(r,p) $  tends to infinity as
 $r$ tends to  $p$.
\end{theorem}
 The  proof given in \cite{We1} relies on 
     properties of $p$-stable processes; it is assumed  that $S$ commutes with an ergodic endomorphism of $(X,\mathcal{ A}, \mu)$, which in fact is unnecessary. The restriction   $p\neq 1$ is only used   at the very end of the proof, but is then crucially necessary. 
  \begin{remark} The pending question of a possible convergence criterion for the space $L^1(\m)$  is of course very interesting. But its true nature is unknown, since we are not operating in a (strictly) stationary context. In particular, $\|S_n(f)-S_m(f)\|_{p,\mu}$,  crucial  in \eqref{metr.comp}, does not even depend on $n-m$ only, in general. 
And we know  
 (see Talagrand \cite[\S\,8.1]{T}), that  a necessary condition for a $1$-stable process to be  sample bounded 
  rather expresses in terms of majorizing measures. This important concept is  however not relevant   in the present context because of its difficulty  of  application.    
 \end{remark}          
  As announced already, we have made the paper self-contained.    We provide proofs of these theorems in Section \ref{s5}. 
 
%%%%%%%%%%%%%%%%%%%%%%%%%%%%%%%%%%%%%%%%%%%%%%%%%%%%%%%%%%%%%%%%%%%%%%%%%%%%%%%%%%%%%%%%%%%%%%%%%%%%%%%%%%%%%%%%%%%%%%  
%%%%%%%%%%%%%%%%%%%%%%%%%%%%%%%%%%%%%%%%%%%%%%%%%%%%%%%%%%%%%%%%%%%%%%%%%%%%%%%%%%%%%%%%%%%%%%%%%%%%%%%%%%%%%%%%%%%%%%  
%%%%%%%%%%%%%%%%%%%%%%%%%%%%%%%%%%%%%%%%%%%%%%%%%%%%%%%%%%%%%%%%%%%%%%%%%%%%%%%%%%%%%%%%%%%%%%%%%%%%%%%%%%%%%%%%%%%%%%  
%%%%%%%%%%%%%%%%%%%%%%%%%%%%%%%%%%%%%%%%%%%%%%%%%%%%%%%%%%%%%%%%%%%%%%%%%%%%%%%%%%%%%%%%%%%%%%%%%%%%%%%%%%%%%%%%%%%%%%  

 \section{Auxiliary Results.}  \label{s4}
\subsection{$\boldsymbol{L^p}$-isometries.} We first   recall a classical result    of Lamperti \cite[Theorem 3.1]{L}.
 Let   $\m$ be   a $\s$-finite measure     on $(X,\A)$. Some basic properties of isometries of $L^p(\m)$  are used in what follows.    Recall that a regular set isomorphism of the measure space
$(X,\mathcal A, \m)$ is a mapping $\Theta$ of $\mathcal A$ into itself such that 
 \begin{eqnarray}\label{iso} {\rm (i)} && \Theta(A^c)=\Theta X\, \backslash \, \Theta A \cr
 {\rm (ii)}& & \Theta\big(\bigcup_{n=1}^\infty A_n\big)= \bigcup_{n=1}^\infty \Theta A_n\quad \text{for disjoint} \ A_n\cr
{\rm (iii)}&& \m(\Theta A)=0\ \text{if  and only if }\ \m(A)=0, \end{eqnarray}
for all elements $A, A_n$ of $\mathcal A$. Then $\Theta$ induces a linear transformation (noted again by $\Theta$) on the space of measurable functions, defined as follows, $\Theta\chi_A= \chi _{\Theta A}$.
\begin{remark}
The question whether a measure preserving set transformation can be obtained from a point mapping has been already considered. By a result of von Neumann,
 so is   the case if for instance $X$ is   a closed region in $\R^n$ and $\m$  is equivalent to the Lebesgue measure, see \cite[p.\! 463]{L}. 
\end{remark}
  \vskip 1 pt 
      \begin{lemma}\label{a0} Let $T$ be a linear operator on  $L^p(\m)$ where $1\le p<\infty$ and $p\neq 2$, and such that $\|Tf\|_{p, \m}=\|f\|_{p, \m}$, for all $f\in L^p(\m)$. Then there exists a regular set-isomorphism $\Theta$ and a function $h(x)$ such that $T$ is given by
    $$ Tf(x)= h(x) \Theta f(x).$$
   Define a measure $\m^*$ by $\m^*(A)= \m(\Theta^{-1} A)$. Then 
   $$ |h(x)|^p=\frac{\dd \m^*}{\dd \m}(x)\qq \text{a.e. on } \ \Theta X\, .$$
      \end{lemma}
 \begin{corollary} \label{c0} Let $\m$ be a probability measure. Let $T$ be a  positive isometry of $L^p(\mu)$ with $1\le p<\infty$ and $p\neq 2$, such that $T1=1$. Then $Tf(x)=  \Theta f(x)$ with $\Theta 1=1$ and $\Theta $ is a regular set-isomorphism. Moreover for any $f\in L^\infty(\m)$, $|Tf|^a\buildrel{\text{a.e.}}\over{=}T|f|^a$, for any $0\le a<\infty$.    Further $|Tf|^p\buildrel{\text{a.e.}}\over{=}T|f|^p$, if $f\in
L^p(\m)$.  \end{corollary} \begin{proof}By Lemma \ref{a0}, $Tf(x)= h(x) \Theta f(x)$. As $\m(X)=1$ and $T1=1$ it follows from the proof of Theorem 3.1 in \cite{L} that   $h(x)\buildrel{\text{a.e.}}\over{=}1$,  and  
$T= \Theta $. But as  $\Theta \chi_A= \chi _{\Theta A}$, we get $|T f|^a=T |f|^a$ for simple functions, for any $0\le a<\infty$.  Hence by approximation $|Tf|^a\buildrel{\text{a.e.}}\over{=}T|f|^a$ holds for all $f\in
L^\infty(\m)$. Further by approximation again, since $T$ is an isometry of $L^p(\mu)$, $|Tf|^p\buildrel{\text{a.e.}}\over{=}T|f|^p$, if $f\in
L^p(\m)$.\end{proof}

 For the sake of completeness, we included the following lemma concerning   the (simpler) case $p=2$.
 %%%%%%%%%
  \begin{lemma}\label{a1}
Let $T$ be a  positive isometry of $L^2(\mu)$ such that $T1=1$.
Then 
  $ (Tf)^2\buildrel{\text{a.e.}}\over{=}Tf^2 $, for any $f\in L^2(\mu)$.

\end{lemma}

\begin{proof} 
 Let $A\in \mathcal{ A}$ with  $0<\mu(A)<1$.
Trivially  $f,g\in L_+^2(\mu)$  have disjoint  supports if  and only if $\|f+g\|^2_{2,\mu}=\|f\|^2_{2,\mu}+\|g\|^2_{2,\mu}$.
Hence it follows that $T{\bf 1}_A$ and $T{\bf 1}_{A^c}$ have disjoint  supports. Let $E=\left\{0<T{\bf
1}_A<1\right\}=\left\{0<T{\bf 1}_{A^c}<1\right\}$. As $E\subset {\rm supp}(T{\bf 1}_A)\cap {\rm supp}(T{\bf 1}_{A^c})$, we conclude that $T{\bf 1}_A$ and
$T{\bf 1}_{A^c}$  are
 indicator  functions. Thus any simple function  is mapped by $T$ into a simple   function. For these functions we have  $(Tf)^2=Tf^2$. 
  Now let   $f\in L_+^2(\mu)$; there exists a sequence $(f_n)\subset
L^\infty(\mu)$ such that $\|f-f_n\|_2 \to 0$ as $n\to \infty$.  First observe by applying twice   H\"older's inequality that 
\begin{align*}
\int_X   (T|f_n^2 - f^2|)^{1/2} \dd \m  \le \big(\int_X    T|f_n^2 - f^2|  \dd \m\big)^{1/2}=\|f_n^2 - f^2\|_1^{1/2}
   \le  \big(\|f_n  - f \|_2 \cdot\|f_n  + f \|_2\big)^{1/2}. \end{align*}
Consequently,
 \begin{align*}
& \| Tf  -\sqrt{Tf^2}\|_1 \le \|T   f  -T f_n  \|_1 + \|    Tf_n -\sqrt{Tf_n^2}  \|_1+\|  \sqrt{Tf_n^2}-\sqrt{Tf^2}\|_1     
\cr &= \|    f  -  f_n  \|_1  +\|  \sqrt{Tf_n^2}-\sqrt{Tf^2}\|_1 \le  \|    f  -  f_n  \|_1  +\big(\|f_n  - f \|_2 \cdot\|f_n  + f \|_2\big)^{1/2}
   \to 0,
\end{align*}
as $n \to \infty$. Hence  $(Tf)^2=f^2$  a.\,e.\,. As $f=f^+-f^-$, we deduce that this holds  for any $f\in L^2(\m)$.  \end{proof}

 \subsection{Stable processes.} \label{secstab} This part was essentially written for  the ergodician reader  not necessarily familiar with stable processes.   We  use very few from the theory.   We refer to \cite{MP}. We also refer the interested reader to the very nice book of Talagrand \cite{T} for a thorough study of the regularity of stable processes.
   For the same reason, the last part of the proof of Theorem \ref{t6} is detailed and we refer to \cite{MP}.  The stable processes we consider are simple,   of finite rank. They are
 however
 {\it not} strongly stationary. Recall and briefly explain some basic facts and  properties of stable random variables and stable processes. \vskip 2 pt Let $0<\a\le 2$.
A real valued random variable
$\theta$  is symmetric
$\a$-stable of parameter $\s$ if 
\begin{equation}\label{vast} \E  e^{it\theta}= e^{-\s^\a |t|^\a}, \qq\quad \forall t\in \R
.\end{equation}
 Then for all $0<r<\a$, $(\E |\theta|^r)^{1/r}= \d(r,\a)\s$,
  where $\d(r,\a)$ depends only on $r$ and $\a$.
 Stable random variables are mixtures of Gaussian random variables. Indeed,  as is well-known  the function
$f(\l)=e^{-\l^\a}$ is completely monotone on
$\R^+$,  for  each $0<\a\le 1$. Consequently, there exists a random variable $v(\a)$ such that 
$\E e^{-\l v(\a)}=f(\l)$, for all $\l\ge 0$.  
Let $\eta(\a):=(2v(\a/2))^{1/2}$.
 Let $g$ be Gaussian standard independent from $\eta$. By taking Fourier
transforms
 $\E e^{i t\eta(\a).g}=\E e^{-t^2\eta(\a)^2/2}=\E e^{-t^2v(\a/2) } =e^{- |t|^\a}$.
Whence it follows that $\theta \buildrel{\mathcal D}\over{=}\eta(\a).g$.
 Let $\theta_1, \ldots, \theta_J$ be i.i.d. $\a$-stable real valued random variables, and let $c_1, \ldots, c_J$ be real numbers. From \eqref{vast} we get
 \begin{equation}\label{sumstable}  \sum_{j=1}^J  c_j\theta_j \buildrel{\mathcal D}\over{=} \theta_1\Big(\sum_{j=1}^J 
|c_j|^\a\Big)^{1/\a}.  \end{equation}
  A stochastic process $\{X(t), t\in T\}$ is a real valued $\a$-stable if any finite linear combination $\sum_{j }   c_jX(t_j)$ is an
$\a$-stable real valued random variable. 
\vskip 8 pt 
From now on, we extend the notation used in \eqref{fj0} in the following way. To any 
$f\in L^p(\m)
$,  $1<p\le \infty$, we  associate  the random element,
\begin{equation}\label{fjgen}F_{J,f}(\o, x) =\frac{1}{J^{1/p}}\sum_{1\le j\le J} \theta_j(\o) T_jf( x) ,\qq \quad \o\in \O, \ x\in X
.\end{equation}
  \begin{remark} \label{fjsplit}As long as entropy criteria are studied in  $L^p(\m)$ with $2\le p\le \infty$, the relevant random elements $ F_{J,f} $ are Gaussian ($\a=2$). When $1<p<2$, we choose them  $p$-stable ($\a=p$).
\end{remark}  Clearly \eqref{fjgen} defines a  real valued $\a$-stable process. It follows in particular  that for any $x\in X$,  
 \begin{equation}\label{fjgen1}F_{J,f}(., x) \buildrel{\mathcal D}\over{=} \theta_1\Big( \frac{1}{J }\sum_{1\le j\le J} |T_jf
(x)|^\a\Big)^{1/\a}  .\end{equation}
  Let     $\{\eta_j,j=1,\ldots , J  \}$ be a sequence
   of   i.i.d.   random variables with the same law than $\eta(\a)$, and let $ \{g_j, j=1,\ldots , J\ \}$ be a sequence
  of i.i.d. Gaussian standard random variables. We assume that these sequences are respectively defined on  joint probability spaces
 $(\Omega',\mathcal{ A'},{\bf P'})$  and $(\Omega '',\mathcal{ A} '',{\bf P} '')$.
\vskip 2 pt
  Then the process    $$
 \mathcal F_{J,f}  (\omega',\omega '',x)= 
 \sum_{j\le  J} \eta_j(\omega ' )g_j(\omega'')T_jf(x) ,   \qq \quad x\in X
 $$
 has the same distribution as  $\{F_{J,f}(x), x\in X\}  $.

%%%%%%%%%%%%%%%%%%%%%%%%%%%%% 
 \subsection{A comparison Lemma.}In the next lemma, we   denote   the norms  corresponding to the
spaces
$L^r(\m)$ and
$L^r(\P)$   respectively by $\|.\|_{r,\mu}$ and
$\|.\|_{r,\P}$.   
 \begin{lemma}\label{a2} Let $1\le p\le 2$.
Let $S_n \colon L^p(\mu)\rightarrow L^p(\mu)$, $n=1,2,\dots $ be
continuous operators verifying assumption {\rm (C)}. 
    Let  $f\in L^p(\mu)$. Let also  $I$ be a finite set of integers such that  
$$\| S_n(f)-S_m(f) \|_{p,\mu} \neq 0,  \qq \quad\hbox{ for all $n\neq m$,   $n,m\in I$}.$$ 

Then  given  any index $\mathcal{ J}_0$  and $0<\e<1$,
   there exists a sub-index
$\mathcal J $ and a measurable set $A $  with  $\mu \{A \} \ge  1-\e$,  such that for all   $x\in A $, we have   for all $J\in \mathcal{ J}$,    all
$n,m\in I$ with  $m\not=n$:
   \vskip 2 pt 
 {\rm (i)} If $1\le p<2$ and $ r<p$, \begin{equation}\label{comp.p}   (1-\e)^{1/p} \le   \frac{   \big\| (S_n-S_m)(F_{J,f} )  (x)\big\|_{r,\P}}{
c(r)\,\|S_n(f)-S_m(f)\|_{p,\mu}}
 \le  (1+\e)^{1/p} , 
 \end{equation} 
  where $c(r)=\|\theta_1\|_r$.
  \vskip 2 pt 
{\rm (ii)} If $p=2$,  
\begin{equation}\label{comp.2}
   (1-\e)^{1/2} \le   \frac{   \big\| (S_n-S_m)(F_{J,f} )  (x)\big\|_{2,\P}}{  \|S_n(f)-S_m(f)\|_{2,\mu}}
 \le  (1+\e)^{1/2}.
  \end{equation} 
   Further, for   any
positive increasing convex function   $G$ on ${\mathbb{R}}^+$, any $J\in \mathcal J $,
\begin{eqnarray*}    \E    G\Big( \sqrt{1-\e}  \sup_{n,m\in I}
Z(S_n(f))-Z(S_m(f))\Big)
 & \le &    \E \int_{X} G\big(\sup_{n,m\in I} S_n(F_{J,f})-S_m(F_{J,f})\big)
d\mu.  \end{eqnarray*}
In particular  for   any $J\in \mathcal{ J} $,
\begin{eqnarray}\label{comp.3}\sqrt{1-\e}\  \E \sup_{n \in I}\, Z(S_n(f))
&\le & \E \int_{X}  \sup_{n \in I} S_n(F_{J,f})  \ d\mu.
\end{eqnarray}\end{lemma} 
 \begin{proof}   We assume $\mathcal{ J}_0
=\N$, the case of an arbitrary index $\mathcal{ J}_0$
being treated identically. 
 \vskip 3 pt {\it Proof of {\rm (i):}} Let
$f\in L^p(\mu)$. By  the commutation assumption,  
$$  S_n(F_{J,f} )= {1\over {J^{1\over p}}}
\sum_{j\leq J}\theta_jS_n( T_j(f) )= {1\over {J^{1/ p}}}
\sum_{j\leq J}\theta_jT_j( S_n(f) ), \qq  \quad  \forall n\ge 1, \ \forall J\ge 1.
$$
 Hence by \eqref{sumstable}, for any  fixed $x\in X$, 
$$
{1\over {J^{1\over p}}}\sum_{j\leq J}\theta_j T_j( S_n(f)-S_m(f) ) (x)
\buildrel {\mathcal{ D}}\over {=} \theta_1 \Big({1\over J}\sum_{j\leq
J}\big|T_j( S_n(f)-S_m(f) ) (x)\big|^p \Big)^{1/ p}.
$$
Using   the fact that  $ |T_jh |^p\buildrel {{\rm a.e.}}\over {=} T_j|h |^p $ if $h\in L^p(\m)$,  
it follows that 
 \begin{eqnarray}\label{nm}
 \E\,   \big| (S_n-S_m)(F_{J,f} )  (x)\big|^r& =& ( \E\,  |\theta_1|^r)\Big({1\over J}\sum_{j\leq
J}\big|T_j( S_n(f)-S_m(f) ) (x)\big|^p \Big)^{r/ p}
\cr &=& ( \E\,  |\theta_1|^r)
 \Big({1\over J}\sum_{j\leq
J}  T_j( \big|S_n(f)-S_m(f)\big|^p) (x)  \Big)^{r/ p},
\end{eqnarray} 
for almost all $x$.  
\vskip 4 pt
   
Let $I$ be a finite set of integers such that  
$$\| S_n(f)-S_m(f) \|_{p,\mu} \ge \d>0,  \qq \quad\hbox{ for all $n\neq m$,   $n,m\in I$}.$$      Let $0<\e <1$ and choose an integer  $L $  sufficiently large so that 
 $2^{-L }\le \e$   and    
$\d\ge 2^{-L-1}/\e $. 
Assumption (C) implies that   
$$
\lim_{J\to\infty}\Big\|{1\over J}\sum_{j\le J}T_j( \big|S_n(f)-S_m(f)\big|^p)-\|S_n(f)-S_m(f)\|^p_{p,\mu}\Big\|_{1,\m}=0,
$$
for all $n,m\in I$. By extraction, we can find an index $\mathcal{ J} =
\{J_k,k>L\} $ (depending on $I$    and $\e$), such that   $$
 \Big\|{1\over J_k} \sum_{j\le J_k} T_j( \big|S_n(f)-S_m(f)\big|^p) -
\|S_n(f)-S_m(f)\|^p_{p,\mu} \Big\|_{1,\mu}\le \frac{1}{\#(I)^22^{ 2k}} ,
$$
 for all $n,m\in I$ and all $k>L$. Put 
 $$ A_k =\Big\{\exists n,m\in I:\Big|{1\over J_k} \sum_{j\le J_k} T_j( \big|S_n(f)-S_m(f)\big|^p) -
\|S_n(f)-S_m(f)\|^p_{p,\mu}\Big|\ge 2^{-k}\Big\}, \qq  k>L. $$ 
By Chebyshev's  inequality,  we have   $\m(A_k)\le 2^{-k} $.
        Let 
\begin{align*}
\qq A_\e(n,m, J) & =   \cr\Big\{
       (1-\e)   \|    S_n(f) &  -S_m(f)\|^p_{p,\mu}   \le {1\over J }  \sum_{j\le J } T_j\Big(  \big|S_n(f)-S_m(f)\big|^p
  \Big)\le ( 1+\e)\|S_n(f)-S_m(f)\|^p_{p,\mu}     \Big\},   \end{align*}  
and 
$$ A_{I,\e }= \bigcap_{  k>L } \bigcap_{n,m\in I  } A_\e(n,m, J_k ).$$
Then, 
\begin{eqnarray*}
 \mu \{A_{I,\e }  \}
 \ge   \mu\Big\{\bigcap_{k>L} A_k^c \Big\}
    \ge  1-\sum_{k>L}2^{-k} 
   = &1-2^{-L}\ge 1-\e.
\end{eqnarray*}

 \vskip 5 pt  As by \eqref{nm}, for any $r<p$,
\begin{eqnarray*} 
    \big\| (S_n-S_m)(F_{J,f} )  (x)\big\|_{r,\P} &=& \|\theta_1\|_r
 \Big({1\over J}\sum_{j\leq
J}  T_j( \big|S_n(f)-S_m(f)\big|^p) (x)  \Big)^{1/ p},
\end{eqnarray*}
it follows that  for every $x\in A_{I,\e }$, we have   $$ 
   (1-\e)^{1/p} \le   \frac{   \big\| (S_n-S_m)(F_{J,f} )  (x)\big\|_{r,\P}}{ \|\theta_1\|_r\,\|S_n(f)-S_m(f)\|_{p,\mu}}
 \le  (1+\e)^{1/p} , $$  
  for all $J\in \mathcal{ J}$,    all $n,m\in I$, $m\not=n$, and $ r<p$.

  \vskip 3 pt {\it Proof of {\rm (ii):}} The proof is the first inequality is identical and so we omit it. Let
$f\in L^2(\mu)$.
Let $0<\e<1$ be fixed.     Let    $I$ be a finite set of integers such that  
$$\| S_n(f)-S_m(f) \|_{p,\mu} \neq 0,  \qq \quad\hbox{ for all $n\neq m$,   $n,m\in I$}.$$ 
    
 Now notice that $\mu \{\  \E  \sup_{n\in I}
S_n(F_{J,f}) \ge 0 \  \}=1$. Using \eqref{comp.2},   next Slepian comparison lemma, we have along the index $\mathcal{ J}$, \begin{eqnarray*}\int_{X}\E  \sup_{n\in
I} S_n(F_{J,f})  \ d\mu &\ge & \int_{A }\E  \sup_{n\in
I} S_n(F_{J,f})  \ d\mu\cr  &\ge &  \sqrt{1-\e}\, \m(A )\, \E  \sup_{n\in I} Z(S_n(f))\ge  (1-\e) \, \E  \sup_{n\in I} Z(S_n(f)). 
\end{eqnarray*}
 Similarly,
  \begin{align*}
&\int_{X} \E
G\big(\sup_{n,m\in I} S_n(F_{J,f})-S_m(F_{J,f})\big)    d\mu
\\
&\qquad\ge
\int_{A } \E G\big(\sup_{n,m\in I} S_n(F_{J,f})-S_m(F_{J,f})\big)
d\mu
\\
&\qquad \ge\int_{A } \E G\big( \sqrt{1-\e}  \sup_{n,m\in I}
Z(S_n(f))-Z(S_m(f))\big) d\m
\\
&\qquad \ge\sqrt{1-\e}  \, \E   G\big(
\sqrt{1-\e}  \sup_{n,m\in I} Z(S_n(f))-Z(S_m(f))\big)
\\
&\qquad
\ge\sqrt{1-\e}  \, \E   G\big( \sqrt{1-\e}  \sup_{n,m\in I}
Z(S_n(f))-Z(S_m(f))\big) .
\end{align*}
  This completes the proof of Lemma \ref{a2}.
\end{proof}
\subsection{Banach Principle.} 
 Let   $$ \mathcal{ Y} = \{f\in L^\infty(\mu) : \|f\|_\infty \le 1 \}.$$ 
 A mapping  $V\colon  (\mathcal{ Y} , d)\rightarrow L^0(\mu)$ is said to be continuous at $0$, if $V$ is $d$-continuous at 0 on $\mathcal{ Y}$.
    When $V$ is linear, then $V$ is continuous at 0   if and only if
$V$ is $d$-continuous on $L^\infty(\mu)$.
   \begin{lemma}[\cite{BJ}]\label{bainfty}
Let $\left\{S_n,n\ge 1\right\}$ be a sequence of linear operators of
$L^\infty(\mu)$ in $L^0(\mu)$. Assume that the following conditions
are realized:
\begin{eqnarray*}{\rm (i)}& &\hbox{Each $S_n$ is continuous at $0$,}
\cr {\rm (ii)}& &\hbox{For any  $f\in L^\infty(\mu)$, $\mu \{ x  :
    \{S_n(f)(x), n\ge 1 \} \text{ converges}   \}=1$.}
\end{eqnarray*}
Then $S^*\colon \mathcal{ Y} \rightarrow L^0(\mu)$ is continuous at $0$.
\end{lemma}
For a short proof, we refer to \cite[p.\,205]{We}.
 The next lemma is used repeatedly. 
  \begin{lemma} \label{a3} Let $(B,\|.\|_B)$ be a Banach space and let $S_n\colon B\to L^0(\m)$, $n\ge 1$, be continuous
in measure operators. Assume that
$$
\m \big\{  \sup_{n\ge 1}|S_n(f)| <\infty \big\}=1,  \qq \qq \text{for all $f\in B$}   .    $$
Then there exists a non-increasing function
$C:]0,1]\to \R_{+}$ such that
for any  $ 0<\varepsilon<1 $, any $J\ge 1$ and any $f\in \mathcal R(B)$, there exists a measurable set 
$X_{\varepsilon,J,f}$ with  $ \m(X_{\varepsilon,J,f}) \ge  1- {\varepsilon} $,   such that for any $x\in X_{\varepsilon,J,f}$, $$\P\Big\{\omega :\sup_{n\ge
1}|S_n(F_{J,f}(\omega,.))(x)\mid
\le C(\e)\E \|F_{J,f} \|_{B}\Big\} \ge
1-\e,$$
    recalling that $F_{J,f}$ is defined in \eqref{fjgen}. 
   \end{lemma} 

 \begin{proof} 
  By the Banach principle, there exists a non-increasing function
$\d:]0,1]\to \R_{+}$ such that
$$ \qquad \m\Big\{\sup_n|S_n(h)|\ge
\d(\varepsilon) \|h\|_{B } \Big\}\le \varepsilon^2/2, \qquad \quad \forall 0<\varepsilon\le 1,\ \forall h\in B. $$
   Let $f\in \mathcal R(B) $. Then by \eqref{tjr}, $F_{J,f}\in B $ almost surely.
   Taking $h=F_{J,f}$ and using   Fubini's theorem, gives
$$\int_X \P\Big\{ \sup_{n\ge 1}|S_n(F_{J,f}) |\ge
\d(\varepsilon)\|F_{J,f} \|_{B}\Big \} \ d\m \le \varepsilon^2/2.$$
Now we bound as follows 
\begin{eqnarray*}& &\int_X \P\Big\{ \sup_{n\ge 1}|S_n(F_{J,f}) |\ge
\frac{2\d(\varepsilon)}{\e^2} \E \|F_{J,f} \|_{B}\Big \} \ d\m\cr &\le &
\int_X \P\Big\{ \sup_{n\ge 1}|S_n(F_{J,f}) |\ge
\frac{2\d(\varepsilon)}{\e^2}  \E \|F_{J,f} \|_{B},   \|F_{J,f} \|_{B}\le \frac{2}{\e^2}\E \|F_{J,f} \|_{B}\Big \} \ d\m
\cr & & +   \P\big\{     \|F_{J,f} \|_{B}> \frac{2}{\e^2}\E \|F_{J,f} \|_{B}\big \}
\cr &\le &
\int_X \P\Big\{ \sup_{n\ge 1}|S_n(F_{J,f}) |\ge
 \d(\varepsilon)       \|F_{J,f} \|_{B} \Big \} \ d\m  +\varepsilon^2/2\ \le \ \varepsilon^2/2+\varepsilon^2/2\ = \ \e^2.\end{eqnarray*}
Hence,
 $$
  \mu\Big\{ x\in X: \P\Big\{\omega :\sup_{n\ge
1}|S_n(F_{J,f}(\omega,.))(x)\mid
\ge \frac{2\d(\varepsilon)}{\e^2}\E \|F_{J,f} \|_{B}\Big\} \ge
\e\Big\}\le \e,$$
or  $$
  \m\Big\{x\in X:\P\Big\{\omega :\sup_{n\ge
1}|S_n(F_{J,f}(\omega,.))(x)\mid
\le \frac{2\d(\varepsilon)}{\e^2}\E \|F_{J,f} \|_{B}\Big\} \ge
1-\e\Big\}\ge 1-\e.
 $$ By letting $C(\e) =\frac{2\d(\varepsilon)}{\e^2}$, we easily conclude.
  \end{proof}

\subsection{Some Gaussian tools.}    The next lemma is  well-known in the theory of Gaussian processes.  
 \begin{lemma}  \label{a4}
   Let $X=\{X_t, t\in T\}$ and $Y=\{Y_t, t\in T\}$ be two centered Gaussian processes  defined on  a finite set $T$.
\vskip 1 pt {\rm(a) [Slepian's Lemma]}      Assume that for any $s,t\in T$,
$$
 \|X_s-X_t\|_2 \le  \|Y_s-Y_t\|_2.
 $$
Then  for any     positive increasing convex
function $f$ on ${\mathbb{R}}^+$, $$\E f\big[\sup_{T\times
T}(X_s-X_t)\big] \le  \E  f\big[\sup_{T\times
T}(Y_s-Y_t)\big]  . $$
In particular,
 $$ \E \sup_{t\in T}X_t  \le  \E \sup_{t\in T}Y_t .   $$

\vskip 1 pt {\rm (b) [Sudakov's minoration]}   There exists a
universal constant $B$ such that for any Gaussian process
$X=\{X_t,t\in T\}$  
 $$ \E \sup_{t\in T}X_t\ge B\inf_{s,t\in T \atop
s\not= t}\|X_s-X_t\|_{2,{\bf P}}\sqrt { \log   \#(T)}. $$
\vskip 1 pt {\rm (c) [Lower bound for Gaussian norms]} Let  $X$ be a Gaussian vector and $N$ a non-negative semi-norm. Then 
$$ \P\{ N(X)<\infty\}=1 \qq \Rightarrow \qq \P\{N(X)\ge \frac{1}{2}\E N(X)\big\}\ge c$$
where $0<c<1$ is a universal constant. 
\vskip 1 pt {\rm (d) [Mill's ratio]}  The Mill's ratio 
$R(x) = e^{x^2/2}\int_x^\infty e^{-t^2/2}\ dt$  verifies   for any $x\ge 0$,
$$
{2\over \sqrt{x^2+4} +x}\le R(x) \le
{2\over \sqrt{x^2+{8\over \pi}} +x}\le \sqrt{{\pi\over 2}}.
 $$
It follows that for any standard Gaussian random
variable $g$, any $T>0$,
$$
\E g^2\chi\{ |g|\ge T \} \le 6e^{-T^2/4}.  
$$ 
\end{lemma}

%%%%%%%%%%%%%%%%%%%%%%%%%%%%%%%%%%%%%%%%%%%%%%%%%%%%%%%%%%%%%%%%%%%%%%%%%%%%%%%%%%%%%%%%%%%%%%%%%%%%%%%%%%%%%%%%%%%%%%  
%%%%%%%%%%%%%%%%%%%%%%%%%%%%%%%%%%%%%%%%%%%%%%%%%%%%%%%%%%%%%%%%%%%%%%%%%%%%%%%%%%%%%%%%%%%%%%%%%%%%%%%%%%%%%%%%%%%%%%  
%%%%%%%%%%%%%%%%%%%%%%%%%%%%%%%%%%%%%%%%%%%%%%%%%%%%%%%%%%%%%%%%%%%%%%%%%%%%%%%%%%%%%%%%%%%%%%%%%%%%%%%%%%%%%%%%%%%%%%  
%%%%%%%%%%%%%%%%%%%%%%%%%%%%%%%%%%%%%%%%%%%%%%%%%%%%%%%%%%%%%%%%%%%%%%%%%%%%%%%%%%%%%%%%%%%%%%%%%%%%%%%%%%%%%%%%%%%%%%  

\section{Proofs.}\label{s5}
As clarified in Remark \ref{fjsplit},   we use the random elements $ F_{J,f} $ introduced in \eqref{fjgen}
 differently, according to the cases $2\le p\le \infty$, in which they are Gaussian, and $1<p<2$, where  we choose them  $p$-stable. This latter case only
concerns the proof of Theorem \ref{t6}. 
\subsection{Proof of Theorem \ref{gen}}    Let $0<\e<1/2$. Let $f\in \mathcal R(B)$.  By Lemma \ref{a3}, there exists a non-increasing function
$C:]0,1]\to \R_{+}$ and 
   a set
 $ X_{\varepsilon,J,f}$ of  measure greater than $  1-\e $ such that for all $x\in  X_{\varepsilon,J,f}$,
$$ \P\Big\{\omega:\sup_{n\ge 1}|S_n(F_{J,f}(\omega,.)(x)|\le
  C(\varepsilon)\E \| F_{J,f}\|_B   \Big\}\ge
1-\e. $$
  It follows from estimate \eqref{fer}  that
\begin{eqnarray*} \E \sup_{n\ge
1}|S_n(F_{J,f}(\omega,))(x)|& \le &{4C(\varepsilon)
  \over 1-\e}\  \E \| F_{J,f}\|_B, \qq\quad \forall x\in X_{\varepsilon,J,f}    . 
  \end{eqnarray*} 
  
Recall that $B\subset L^2(\m)$. Let $I$ be a finite set of integers such that
$\|S_n(f)-S_m(f)\|_{2 }\not= 0 $,  for all   $m,n\in
I$, $m\neq n$. By Lemma \ref{a2}-(ii), taking $ \mathcal J_0=\N $, there exists a sub-index
$\mathcal{ J}$ such that if
$$
A(I) =  \Big\{ \forall J\in \mathcal{ J}, \ \forall n,m\in I,\
m\not=n, \    { \|  S_n(F_{J,f})-S_m(F_{J,f}) \|_{2,\P}
\over   \| S_n(f)-S_m(f) \|_{2,\mu}} \ge \sqrt{1-\e}\,
\Big\},$$ 
 then   $\mu\left\{A(I)\right\} \ge  \sqrt{1-\e} $.

 By integrating on $ X_{\varepsilon,J,f}\cap A(I)$,  next using the fact that $\E \sup_{n \in I} S_n(F_{J,f})\ge 0$, and  Lemma \ref{a4}-(a),  we get  for any $J\in \mathcal{ J}$,
\begin{eqnarray*} \int_{X_{\varepsilon,J,f} } \E \sup_{n \in I} S_n(F_{J,f}) 
\dd\mu&\ge & \int_{X_{\varepsilon,J,f}\cap A(I)} \E \sup_{n \in I} S_n(F_{J,f})  \dd\mu
\cr & \ge & \sqrt{1-\e}\, \m\{ X_{\varepsilon,J,f}\cap A(I)\}\  \E  \sup_{n \in I} Z(S_n(f))
 \cr & \ge & \sqrt{1-\e}\,(    \sqrt {1-\e}-  \e)\  \E  \sup_{n \in I} Z(S_n(f))
  \cr & \ge &  (   1-2\e)\  \E  \sup_{n \in I} Z(S_n(f))      .
\end{eqnarray*}
  By combining, for any $J\in \mathcal{ J}$, 
 \begin{eqnarray}\label{comb1}
    \E  \sup_{n \in I} Z(S_n(f))
&\le & \frac{1}{  (1-2\e)}\,  \int_{X_{\varepsilon,J,f} } \E \sup_{n \in I} S_n(F_{J,f})  \dd\mu  
\ \le \  K(\e) \E \| F_{J,f}\|_{B}  , \end{eqnarray}
with  
$$K(\e)=\frac{4C(\varepsilon)
 }{  (1-2\e)( 1-\e)    }\,.$$
Therefore, for any $f\in \mathcal R(B)$, any finite set $I$,
$$\E \sup_{n\in I} Z(S_n(f)) \le K(\e) \inf_{J\in \mathcal J} \sup_{H\ge  J}  \E \| F_{H,f}\|_{B}  \ =\ K(\e) \,\limsup _{H\to \infty}
  \E \| F_{H,f}\|_{B} .$$
 Taking $I=[1,N]$ and letting next $N$ tends to infinity, gives
 \begin{eqnarray*} \E \sup_{n \ge 1} Z(S_n(f)) &\le &  K(\e) \,\limsup _{H\to \infty}  \E \| F_{H,f}\|_{B} .
  \end{eqnarray*} 
  
\subsection{Proof of Theorem \ref{t3}}        Let $f\in L^\infty(\m)$.  Fubini's theorem and Lemma \ref{a1}   allow us to write,\begin{align*}
\E \int|F_{J,f} |^pd\mu&\le C_p^p  \int
\big(\E |F_{J,f} |^2\big)^{p/2} d\mu  =
C_p^p\int \Big({1\over J}\sum_{j\le
J}T_jf^2(x)\Big)^{p/2}d\mu(x).
\end{align*}
 By assumption
$$
\lim_{J\to\infty}\Big\|{1\over J}\sum_{j\le J}T_jf^2-\|f\|^2_{2,\mu}\Big\|_{1,\m}=0.
$$
By proceeding by extraction,   this convergence  also holds almost surely along some subsequence $\mathcal{ J}_0$. As  ${1\over J}\sum_{j\le
J}T_jf^2(x)\le \|f\|_{\infty}^2$, we further deduce from the   dominated convergence theorem,
$$
\lim_{\mathcal{ J}_0\ni J\to \infty}\E \int\Big({1\over J}\sum_{j\le
J}T_jf^2(x)\Big)^{p/2}d\mu = \|f\|^p_{2,\mu}.
$$
Let $0<\e<1/2$. Extracting if necessary from $ \mathcal{ J}_0$ a sub-index which
we call again $ \mathcal{ J}_0$, we   thus conclude that
\begin{equation*}
\E
\|F_{J,f}\|_{p,\mu}\le (1+\e)C_p\|f\|_{2,\mu},\qq \qq \forall J\in  \mathcal{
J}_0.
\end{equation*}
Next the proof is exactly the same as before except that we replace everywhere the norm $\|.\|_{B}$ by the norm $\|.\|_{p,\mu}$. Let $I$ be a finite
set  of integers. From   Lemma \ref{a2}, we can extract from $ \mathcal{ J}_0$    a partial index ${\mathcal J}$ such that the analog of \eqref{comb1}
holds, namely 
  for any $J\in \mathcal{ J}$, 
 \begin{eqnarray*} 
    \E  \sup_{n \in I} Z(S_n(f))
&\le &   K(\e) \, \E \| F_{J,f}\|_{p,\mu} \ \le \ C_pK(\e)(1+\e)   \| f\|_{p,\mu}  . \end{eqnarray*}
It suffices now to give   an  explicit value to $\e$. A simple approximation argument     allows  to  get the same inequality for all $f\in L^2(\m)$.   
Sudakov's minoration further implies
$$\sup_{\varrho >0} \varrho \sqrt{  \log
N_f(\varrho)}\ \le \ C_pK(\e)(1+\e)   \| f\|_{2, \m} .$$

 %%%%%%%%%%%%%%%%%%%%%%%%%%%%%%%%%%%%%%%%%%%%%%%%%%%%%% 
\subsection{Proof of Theorem \ref{t5}}   (i)    By Lemma \ref{a2}-(b), given any index $\mathcal{ J}_0$,  there exists an index
$\mathcal{ J}\subseteq \mathcal{ J}_0$ such that for any $J\in \mathcal{ J}$,
$$
(1-\e) \E \sup_{n\in I} Z(S_n(f))
\le \E \int \sup_{n\in I}  S_n(F_{J,f})  \ d\mu .
$$ 
 And for   any
positive increasing convex function   $G$ on ${\mathbb{R}}^+$, any $J\in \mathcal J$,
\begin{eqnarray*}    \E    G\Big( \sqrt{1-\e}  \sup_{n,m\in I}
Z(S_n(f))-Z(S_m(f))\Big)
 & \le &    \E \int_{X} G\big(\sup_{n,m\in I} (S_n-S_m)(F_{J,f}) \big)
d\mu.  \end{eqnarray*}
 
In the following calculation we put
$$L= \sup_{
   \|g\|_{B}\le 1  }\int
\sup_{n\in I} |S_n(g)| \ d\mu, $$
and we let     $u_0=0$, $u_n= \e(1+\e)^{n-1}$
$n\ge 1$. Then
\begin{eqnarray*} 
& &\E \int \sup_{n\in I} \left| S_n(F_{J,f})\right|  \ d\mu  =\sum_{k=1}^\infty \E
\left({\bf 1}_{u_{k-1}\le \|F_{J,f}\|_{B}<u_k} \cdot \int
\sup_{n\in I} \left| S_n(F_{J,f})\right| \ d\mu \ \right)
 \cr &  \le &\sum_{k=1}^\infty  \P\big\{ u_{k-1}\le \|F_{J,f}\|_{B}<u_k\big\} \sup_{u_{k-1}\le
\|g\|_{B}<u_k  }\int \sup_{n\in I} |S_n(g)| \ d\mu
 \cr &  \le &
L\sum_{k=1}^\infty
u_k\P\big\{u_{k-1}\le \|F_{J,f}\|_{B}<u_k\big\}
 \cr &  \le & L\big( u_1\P \{  \|F_{J,f}\|_{B}<u_1 \} +
(1+\e)\E   \|F_{J,f}\|_{B} \cdot {\bf 1}_{u_1\le
\|F_{J,f}\|_{B} }\big)
 \cr &  \le & L  ( \e +
(1+\e)\E    \|F_{J,f}\|_{B} )
 . \end{eqnarray*}

 By combining, and letting  next $\e$ tends  to $0$, we get for any $f\in
L^2(\m)$,
$$
\E \sup_{n\in I} Z(S_n(f)) \le    \E    \|F_{J,f}\|_{B}  \sup_{
   \|g\|_{B}\le 1  }\int
\sup_{n\in I} |S_n(g)| \ d\mu.$$
Similarly
\begin{align*} \qq     \E    G\Big( \sqrt{1-\e}  \sup_{n,m\in I}
\big|Z(S_n(f))-Z(S_m(f))\big|\Big)\qq\qq\qq\qq \qq\qq\qq\qq 
 \cr      \le     \E    \|F_{J,f}\|_{B}  \sup_{
   \|g\|_{B}\le 1  }\  \E \int_{X} G\big(\sup_{n,m\in I} \big|(S_n-S_m)(g) \big|\big)
d\mu.   \end{align*}
(ii) Let $B= L^p(\m)$. We have seen that there exists an index $ \mathcal{ J}_0$ such that \begin{equation*}
\E
\|F_{J,f}\|_{p,\mu}\le (1+\e)C_p\|f\|_{2,\mu},\qq \qq \forall J\in  \mathcal{
J}_0.
\end{equation*}
Therefore, by letting $J$  tend to infinity along $\mathcal{
J}_0$, next $\e$ tend to zero, we get
 $$
\sup_{\|f\|_{2,\mu}\le 1}\E \sup_{n\in I} Z(S_n(f)) \le    C_p     \sup_{
   \|g\|_{p,\mu}\le 1  }\int
\sup_{n\in I} |S_n(g)| \ d\mu.
$$
And 
\begin{align*} \qq  \sup_{\|f\|_{2,\mu}\le 1}   \E    G\Big(   \sup_{n,m\in I}
\big|Z(S_n(f))-Z(S_m(f))\big|\Big)\qq\qq\qq\qq  
 \cr      \le    C_p \sup_{
   \|g\|_{p,\m}\le 1  }\  \E \int_{X} G\big(\sup_{n,m\in I} \big|(S_n-S_m)(g) \big|\big)
d\mu.   \end{align*}

%%%%%%%%%%%%%%%%%%%%%%%%%%%%%%%%%%%%%%%%%%%%%%%%%%%%%%%
 
\subsection{Proof of Theorem \ref{t6}}      Let  $f \in L^p(\mu)$. Let $J $ be any positive integer   
    and $x\in X$. By \eqref{fjgen1},
$$
{1\over {J^{1\over p}}}\sum_{j\leq J}\theta_j T_jf(x)
\buildrel {\mathcal{ D}}\over {=} \theta_1 \Big({1\over J}\sum_{j\leq
J}|T_jf(x)|^p \Big)^{1/ p}.
$$
 Thus for any $r<p$,
$$
 \E\,   |F_{J,f}  (x)|^r = ( \E\,  |\theta_1|^r)
 \Big({1\over J}\sum_{j\leq J}|T_jf(x)|^p\Big)^{r/ p} .
$$ 
By Corollary \ref{c0},  $ |T_jf(x) |^p\, \buildrel{\rm a.e.}\over{=} \,  T_j|f |^p  (x) $, 
so that  we have \begin{equation} \label{r}
 \E\,   |F_{J,f}  (x)|^r = ( \E\,  |\theta_1|^r)
 \Big({1\over J}\sum_{j\leq J} T_j | f  |^p(x)\Big)^{r/ p} .
\end{equation} 
  for almost all  $x$  and   all $J\ge 1$.  As trivially $  T_j | f  |^p\in L^1(\m)$,
  we deduce
\begin{equation} \label{r}
 \E\,  \int_X |F_{J,f}  (x)|^r\dd \m(x)  = ( \E\,  |\theta_1|^r)\int_X 
 \Big({1\over J}\sum_{j\leq J} T_j | f  |^p(x)\Big)^{\frac{r}{p}} \dd \m(x).
\end{equation}   
Hence,
  \begin{align*}   \Big|\E\,  \int_X |F_{J,f}  (x)|^r\dd \m(x)-& ( \E\,  |\theta_1|^r)\|f\|_{p,\m}^r \Big| 
\cr  = \ &
( \E\,  |\theta_1|^r)\Big|\int_X 
 \Big({1\over J}\sum_{j\leq J} T_j | f  |^p(x)\Big)^{\frac{r}{p}} \dd \m(x) -  (\|f\|_{p,\m}^p)^{\frac{r}{p}} \Big|
\cr \le \ &
( \E\,  |\theta_1|^r) \int_X 
\Big| \Big({1\over J}\sum_{j\leq J} T_j | f  |^p(x)\Big)^{\frac{r}{p}}   -  (\|f\|_{p,\m}^p)^{\frac{r}{p}} \Big|\dd \m(x) 
\cr \le \ &
( \E\,  |\theta_1|^r) \int_X 
\Big|  {1\over J}\sum_{j\leq J} T_j | f  |^p(x)     -   \|f\|_{p,\m}^p   \Big|^{\frac{r}{p}}\dd \m(x) 
\cr \le \ &
( \E\,  |\theta_1|^r) \Big(\int_X 
\Big|  {1\over J}\sum_{j\leq J} T_j | f  |^p(x)     -   \|f\|_{p,\m}^p   \Big| \dd \m(x)\Big)^{r/p}
\cr        &\!\! \to 0 ,\end{align*}
as $J$ tends to infinity  by  assumption (C).
Therefore, $$
\lim_{J\to\infty}   \E\,  \int_X |F_{J,f}  (x)|^r\dd \m(x)= ( \E\,  |\theta_1|^r)\|f\|_{p,\m}^r   , \qq \quad   \forall
0<r<p.
$$
  By using H\"older's
inequality, we deduce that 
 \begin{equation}\label{fr}
\E
\|F_{J,f}\|_{r,\mu}\le \Big(\E \int    |F_{J,f}  (x)|^r \dd  \mu\Big)^{1/r}\le 2\|\theta_1\|_r \, \|f\|_{p,\mu}, 
\end{equation}
for all   $J\ge J_0$, say.  
\vskip 3 pt 
 By assumption, property ($\mathcal B_r$) holds for some $1<r<p$.   From Lemma \ref{a3} follows that there exists a non-increasing function
$C:]0,1]\to \R_{+}$ such that for any $f\in L^r(\m)$, for  any $J\ge 1$, any $0<\e<1 $, there exists 
   a mesurable set
 $ \widetilde X=  \widetilde X_{\varepsilon,J,f}$ of  measure greater than $  1-\sqrt{\varepsilon} $, such that for all $x\in  \widetilde X$,
 \begin{equation}\label{maxr} \P\big\{\omega:\sup_{n\ge 1}|S_n(F_{J,f}(\omega,.)(x)|>  C(\varepsilon) \| F_{J,f}\|_{r,\mu}   \big\}\le \e. 
\end{equation} 
We assume   $0<\e<1/6$ in what follows.   
  Let $\d(\e)=C(\e)/\e$. Let also $x\in \widetilde X$, $J\ge J_0$. Using   Chebyshev's inequality and \eqref{fr}, we get
     \begin{eqnarray*} & &\P\big\{\omega:\sup_{n\ge 1}|S_n(F_{J,f}(\omega,.)(x)|> 2 \d(\e) \|\theta_1\|_r \, \|f\|_{p,\mu}  \big\}  \cr &\le & \ \P\big\{\omega:\sup_{n\ge 1}|S_n(F_{J,f}(\omega,.)(x)|>  \d(\e)\E \| F_{J,f}\|_{r,\mu}   \big\} \le \  \P\big\{\| F_{J,f}\|_{r,\mu}> \E \| F_{J,f}\|_{r,\mu}/  \e\big\}
 \cr & &\quad + \P\big\{\omega:\sup_{n\ge 1}|S_n(F_{J,f}(\omega,.)(x)|>  \d(\e)\E \| F_{J,f}\|_{r,\mu}  , \| F_{J,f}\|_{r,\mu}\le  \E \| F_{J,f}\|_{r,\mu}/  \e \big\}
 \cr &\le &   \e + \P\big\{\omega:\sup_{n\ge 1}|S_n(F_{J,f}(\omega,.)(x)|>  C(\varepsilon)    \| F_{J,f}\|_{r,\mu}  \big\} \cr &\le & 2\e. \end{eqnarray*} 
   
   Therefore,    \begin{equation}\label{maxr1}
\P\big\{\omega:\sup_{n\ge 1}|S_n(F_{J,f}(\omega,.)(x)|\le   2\d(\e) \|\theta_1\|_r \, \|f\|_{p,\mu}  \big\}\ge 1-2\e, \qq \quad \forall x\in  \widetilde X ,  \ \forall J\ge J_0 .  
\end{equation}

    \vskip 2 pt

Let $\d$ be some fixed positive real. Let $I$ be a finite  set of positive integers and let  $M=\#\{I\}$. Assume that 
   $ \|S_n(f)- S_m(f) \|_{p,\mu}\ge \d$ if $n\neq m$, $n,m\in I$.  
   By Lemma \ref{a2}-(i),    there exists an index
$\mathcal J  $ and a measurable set $A =A_{\varepsilon,I,f}$  such that   $\mu \{A \} \ge  1-\e$, and further,   for all $x\in
A $, the following inequalities  $$ 
   (1-\e)^{1/p} \|\theta_1\|_r\,\|S_n(f)-S_m(f)\|_{p,\mu} \le      \big\| (S_n-S_m)(F_{J,f} )  (x)\big\|_{r, \P}
 \le  (1+\e)^{1/p}\|\theta_1\|_r\,\|S_n(f)-S_m(f)\|_{p,\mu}, $$  
are satisfied  for all $J\in \mathcal J$,  all $n,m\in I$ and all $ r<p$. Set 
$$Y=Y_{\varepsilon,I,J,f}= \widetilde X\cap A  .$$

For each $x$ fixed, the process
 $$ S_{J,f,x}(\o,  n )= {1\over {J^{1/ p}}}
\sum_{j\le   J}\theta_j(\o) T_j  S_n(f)(x) , \qq  \quad  n\ge 1,
$$
is a $p$-stable random function. Further, the process 
   $$
 \mathcal S_{J,f,x}  (\omega',\omega '',n)=\frac{1}{J^{1/p}} 
 \sum_{1\le j\le  J} \eta_j(\omega ' )g_j(\omega'')T_jS_nf(x) ,   \qq \quad n\ge 1
 $$
 has the same distribution as  $\{S_{J,f,x}(., n), n\ge 1\}  $.       Recall (sub-section \ref{secstab}) that we   have underlying     joint probability spaces
 $(\Omega',\mathcal{ A'},{\bf P'})$  and $(\Omega '',\mathcal{ A} '',{\bf P} '')$ on which the sequence   
$\{\eta_j,j\ge 1  \}$  and the sequence $ \{g_j, j\ge 1 \}$   
  of i.i.d. Gaussian standard random variables are respectively defined. Here we take both sequences infinite. 
\vskip 2 pt
Thus  \eqref{maxr1} reads:  for all $x\in \widetilde X $, and all $  J\ge J_0$,
\begin{eqnarray} \label{maxr2}
  \P' \times \P''\Big\{(\omega',\omega ''): \sup_{n\ge 1} |\mathcal S_{J,f,x}(\omega',\omega '',n)|\le   2\d(\e) \|\theta_1\|_r \, \|f\|_{p,\mu} 
 \Big\} \ge 1-2 \varepsilon .  
\end{eqnarray}
Let 
$$ H(\o')= \P''\big\{ \omega '':    \sup_{n\ge 1} |\mathcal S_{J,f,x}(\omega',\omega '',n)|\le   2\d(\e) \|\theta_1\|_r \, \|f\|_{p,\mu}  \big\}.$$
By Fubini's theorem, the left-term in \eqref{maxr2} also writes
\begin{eqnarray*} 
 \int_{\O'} H(\o') d\P'(\o')&=& \int_{\o': H(\o')\le \varepsilon } H(\o') d\P'(\o')+ \int_{\o': H(\o')> \varepsilon } H(\o') d\P'(\o')
\cr & \le & \varepsilon + \P'\{\o': H(\o')>\varepsilon  \}.  
\end{eqnarray*}Hence
\begin{eqnarray} \label{maxr3}
  \P' \Big\{ \omega ':
\P''\big\{ \omega '':   \sup_{n\ge 1} |\mathcal S_{J,f,x}(\omega',\omega '',n)| \le
2\d(\e) \|\theta_1\|_r \, \|f\|_{p,\mu}  \big\} \ge   \varepsilon\Big\} \ge 1-3  \varepsilon.
 \end{eqnarray}
 
 For each fixed $\o'\in \O'$,  $\{\mathcal S_{J,f,x}(\omega',.,n), n\ge 1\}$ is a Gaussian process. 
\noi Let $  \E_{ \P''}$ denote   the expectation symbol with
respect to   $\P''$. By using estimate \eqref{fer}, for every $x\in X_{\varepsilon,J,f}$,
\begin{equation}\label{gaussest}  1-3  \varepsilon \leq {\P} '\Big\{ \omega ':
   \E_{ {\P''} }\, \sup_{n\ge 1} |\mathcal S_{J,f,x}(\,\cdot,\omega ',x))|
  \le \textstyle {8\d(\e)  \over \e}\|\theta_1\|_r
\|f\|_{p,\mu}  \Big\}. 
\end{equation} 
 Write for a while 
\begin{eqnarray*}D(\omega ,n,m)&=  &D_{J,f,x}(\omega ,n,m)\ =\ S_{J,f,x}(\omega ,n)- 
S_{J,f,x}(\omega ,m) 
\cr  \mathcal D(\omega',\omega '',n,m) &=  &\mathcal D_{J,f,x}(\omega',\omega '',n,m)\ =\  \mathcal S_{J,f,x}(\omega',\omega
'',n)-\mathcal S_{J,f,x}(\omega',\omega '',m) 
\cr \D(  n,m) &=  &\D_{J,f,x}(  n,m)\ =\ \Big(\frac{1}{J} \sum_{1\le j\le J} \big|T_j(S_n
-S_m)f(x)\big|^p\Big)^{1/p} .
\end{eqnarray*}
By   \eqref{sumstable}, 
\begin{equation*}  
 \E_{\P'}\E_{\P''} e^{it \mathcal D(\omega',\omega '',n,m)} = \E_\P e^{it (S_{J,f,x}(. ,n)- 
S_{J,f,x}(.,m))} =  \E_\P
e^{it \theta_1 \D(  n,m)}= e^{-|t|^p\D(  n,m)^p }.\end{equation*}   
As  $\E e^{it g}= e^{-t^2\tau^2/2}$ where 
$\tau =  (\E g^2)^{1/2}$, we get  
  from \eqref{vast}, 
\begin{equation*} 
 \E_{\P'}\E_{\P''} e^{it \mathcal D(\omega',\omega '',n,m)} = \E_{\P'}e^{-t^2\|\mathcal D(\omega',.,n,m)\|^2_{2,\P''}/2}  = e^{-|t|^p\D(  n,m)^p
}.\end{equation*}   
 Put for each $\o'\in \O'$, 
$$ d_{J,\omega ',x}(n,m)=\|\mathcal D_{J,f,x}(\omega',.,n,m)\|_{2,\P''}. 
$$
And let 
$$ d_{J, x}(n,m)= \big(\frac{1}{J} \sum_{j\le J} T_j|S_n(f)-S_m(f)|^p(x)\big)^{1/p}.  $$ 
We note that $ d_{J, x}(n,m)= \D_{J,f,x}(  n,m)$ for almost all $x\in X$. Further
\begin{equation*} 
   \E_{\P'}e^{-t^2d_{J,\omega ',x}(n,m)^2  /2}  = e^{-|t|^pd_{J, x}(n,m)^p
}.\end{equation*} 
Then
 \begin{eqnarray*}\P\big\{ \exists n,m\in I:  d_{J,\omega ',x}(n,m)<\e d_{J, x}(n,m)\big\}&\le & \sum_{n,m\in I}\P\big\{  
 e^{-t^2  d_{J,\omega ',x}(n,m)/2}>e^{-t^2\e^2 d^2_{J, x}(n,m)}\big\}
\cr &\le & M^2 e^{ t^2\e^2 d^2_{J, x}(n,m)- |t|^pd_{J, x}(n,m)^p}, \end{eqnarray*}
and so,\begin{eqnarray*}\P\big\{ \exists n,m\in I:  d_{J,\omega ',x}(n,m)<\e d_{J, x}(n,m)\big\}&\le &  M^2\inf_{t>0}\, e^{ t^2\e^2 d^2_{J,
x}(n,m)- |t|^pd_{J, x}(n,m)^p} . \end{eqnarray*} 
 The function $\p(t)= e^{t^2a-t^pb}$ has an extremum at the value $t^*=  \big(\frac{pb}{2a}\big)^{\frac{1}{2-p}}$, and 
$$\p(t^*)=
\exp\big\{a^{-\frac{p}{2-p}}b^{ \frac{2}{2-p}}\big[(\frac{p }{2} )^{\frac{2}{2-p}}-(\frac{p }{2} )^{\frac{p}{2-p}}\big]\big\}.$$
Applying this with $a=\e^2 d^2_{J, x}(n,m) $, $b=d_{J, x}(n,m)^p$, we get
\begin{eqnarray*}& &\P\big\{ \exists n,m\in I:  d_{J,\omega ',x}(n,m)<\e d_{J, x}(n,m)\big\}\cr &\le &  M^2\exp\big\{ \e  
 ^{-\frac{2p}{2-p}}(  d _{J, x}(n,m))^{-\frac{2p}{2-p}}d_{J, x}(n,m)^{
\frac{2p}{2-p}}\big[(\frac{p }{2} )^{\frac{2}{2-p}}-(\frac{p }{2} )^{\frac{p}{2-p}}\big]\big\}
\cr &:=&M^2\exp\big\{- \e   ^{-\frac{2p}{2-p}} C(p)\big\} . 
\end{eqnarray*} 
with $C(p)=  (\frac{p }{2} )^{\frac{p}{2-p}}- (\frac{p }{2} )^{-\frac{2}{2-p}}  >0$.
 Choose $\e=(\tau\log M)^{-\frac{2-p}{2p} }$. We get
\begin{eqnarray} \label{minoration}\P\big\{ \exists n,m\in I:  d_{J,\omega ',x}(n,m)<(\tau\log M)^{-\frac{2-p}{2p} } d_{J, x}(n,m)\big\} 
&\le &M^{2-\tau C(p)} \le\frac{1}{2} , \end{eqnarray} 
for $\tau = \tau(p)$ depending on $p$ only, and small enough.

 Now if 
$x\in Y $, we have  
$$ 
        \big\| (S_n-S_m)(F_{J,f} )  (x)\big\|_{r, \P'\times \P''}
 \ge c(\e, r,p)\,\|S_n(f)-S_m(f)\|_{p,\mu}  ,     $$ 
 for all $J\in \mathcal{ J}$,     all $n,m\in I$,   $m\not=n$,  and all $ r<p$. 
As 
 $(S_n-S_m)(F_{J,f} )  (x)\buildrel{\mathcal D}\over{=} \mathcal (\mathcal S_{J,f,x}( n)-(\mathcal S_{J,f,x}( m)) $,
we have 
$$      \big\| (S_n-S_m)(F_{J,f} )  (x)\big\|_{r, \P'\times \P''}= \|\theta_1\|_{r}\, d_{J, x}(n,m),$$ 
whence 
\begin{eqnarray} \label{metr.comp}
      d_{J, x}(n,m)
& \ge & c(\e, r,p)\,\|S_n(f)-S_m(f)\|_{p,\mu} ,   
 \end{eqnarray}
 for all $J\in \mathcal{ J}$,     all $n,m\in I$,   $m\not=n$. 

  Putting together \eqref{minoration}  and \eqref{gaussest} implies  that there exists a measurable set   $\Omega '_0$ with 
  ${\P} ' (\Omega '_0 )>0$,    such that for any $\omega '\in \Omega '_0$, and all $n,m\in I$,   
$$
 d_{J,\omega ',x}(n,m)\ge c(\e, r,p)\frac{d_{J, x}(n,m)}{(\log \#\{I\})^{1/p-1/2}}\ge c(\e, r,p) \frac{\d}{(\log \#\{I\})^{1/p-1/2}}  $$
By Sudakov's inequality,  
\begin{equation}\label{sud}  \|f\|_{p,\mu}  \ge  c(r,p)   \E_{ {\P''} }\, \sup_{n\in I} |\mathcal S_{J,f,x}(\,\cdot,\omega ',x))|\ge  c(r,p) \d  (\log
\#\{I\})^{1/2+1/2 -1/p}.
\end{equation}
 A routine argument together with \eqref{gaussest} now easily    leads to 
 $$
 \|f\|_{p,\mu}
\ge  c(r,p)\, \sup_{\d>0}\d \big(\log N^p_f(\d)  \big)^{1/ q},
 $$
where  $ c(r,p)>0$ depends on $r$ and $p$ only. It is only at this last stage that the fact that $p>1$ is necessary.

\vskip 5 pt

 %%%%%%%%%%%%%%%%%%%%%%%%%%%%%%%%%%%%%%%%%%%%%%%%%%%%%%% 
\subsection{Proof of Theorem \ref{t2}}  
Let $f\in L^\infty(\mu)$ such that $\|f\|_{2, \m}=1$.  Let $I$ be a finite subset of $\N$ and let $M= \#\{I\}$.  
 Write for a while $N=N_\o=|S_n(F_{J,f}(\o,.))|$, $\b(\o)=\m \big\{ x: N_\o(x)\ge \frac{1}{2}\E N_\o(x)\big\} $. By Lemma \ref{a4}-(c),   for each $x$,
 $$ \P\big\{N_\o(x)\ge \frac{1}{2}\E N_\o(x)\big\}\ge c.$$
 And so, 
$$ \m\otimes\P\big\{(\o, x): N_\o(x)\ge \frac{1}{2}\E N_\o(x)\big\}  \ge  c.$$
  We have  
$$ c\le \E \b= \E \b\big(\chi_{\{\b \ge c/2\}}+\chi_{\{\b \le c/2\}} \big)\le c/2+\P\{ \b\ge c/2\}.   $$
Hence $\P\{ \b\ge c/2\}\ge c/2$, and using the previous notation, we  deduce that for each $J\ge 1$,  there exists a measurable set $D_J$ of probability  larger than $c/2$, such that we have 
\begin{eqnarray*}  \m \big\{x: |S_n(F_{J,f}(\o,.))(x)|\ge \frac{1}{2}\E |S_n(F_{J,f})(x)|\big\}\ge   c /2  , \qq \forall \o\in D_J. 
\end{eqnarray*} 
 
  Let $0<\gamma <1$ be fixed.  By   Lemma \ref{a2}-(ii), there exists an index
$\mathcal J $ and a measurable set $A $  with  $\mu \{A \} \ge  \g^2$, and such that for all   $x\in A $, we have 
\begin{eqnarray}\label{debut} 
\gamma \E  \sup_{n\in I} Z(S_n(f)) &\le &\E \int \sup_{n\in I}
S_n(F_{J,f})\ d\mu \qq \quad \forall J\in \mathcal{ J}.  
\end{eqnarray} 
\noi
 Hence,
\begin{eqnarray*}
*\label{debut1} 
\m \big\{x: |S_n(F_{J,f}(\o,.))(x)|\ge \frac{\gamma}{2}  \E  \sup_{n\in I} Z(S_n(f))\big\}\ge   c /3  , \qq \forall \o\in D_J,
\end{eqnarray*}
  assuming $\g$ sufficiently close to 1 and   all $J\in \mathcal{ J}$ greater than some sufficiently large number, which we do. 
  
  We simplify the notation  in what follows and   write   
$F_J=F_{J,f}.$ Put  for any $A>0$,
$$
E_A=\big\{(\omega,x)\in \Omega\times X  :  |F_J(\omega,x)|\le
A\big\} ,\qq \qq   E_{A,\omega}= \big\{x\in X  :  (\omega,x)\in
E_A\big\},
$$
 and let  for any $ \omega \in \Omega$, $x\in X$,
$$ F_{A,J}(x)=F_{A,J,\omega}(x)=F_J(\omega,x)\cdot{\bf 1}_{E_{A,\omega}}(x), \qq F^{A,J}(x)=F^{A,J,\omega}(x)=F_J(\omega,x)\cdot{\bf 1}_{E_{A,\omega}^c}(x).
 $$
  Obviously,
\begin{equation}\label{basic}\E  \int \sup_{n\in I}
 S_n(F_{J,f})\ d\mu  \le \E \int \sup_{n\in I}|S_n(F^{A,J })|\
 d\mu + \E  \int \sup_{n\in I}|S_n(F_{A,J })|\
 d\mu. 
 \end{equation}
By definition $F_{A,J,\omega}(\,\cdot\,)$ (resp.\
$F^{A,J,\omega}(\,\cdot\,)$) is $\mathcal{ A}$-measurable. As $f\in
L^\infty(\mu)$, we have
$$
\P\big\{ \omega    :    F_{A,J,\omega}(\,\cdot\,)\  {\rm
and}\  F^{A,J,\omega}(\,\cdot\,)   \in L^\infty (\mu)\big\}=1.
$$
As $\max_{i\le n} x_i\le (\sum_{i\le n} x_i^2)^{1/2}$ for any nonnegative real numbers, by using twice Cauchy-Schwarz's inequality, next Fubini's
inequality, we get
 \begin{eqnarray}\label{fa}\E \int\! \sup_{n\in I}|S_n(F^{A,J })|\ d\mu & \le &\E \Big(
\sum_{n\in I} \int\! |S_n(F^{A,J })|^2\ d\mu \Big)^{{1/ 2}}\! \cr&\le&  \Big(
\sum_{n\in I} \int\! \E |S_n(F^{A,J })|^2\ d\mu \Big)^{{1/ 2}}  
  \, \le \,
\sqrt{M}  \, \E \|F^{A,J }\|_{2,\mu} .  
\end{eqnarray}
We have to estimate $\|F^{A,J}\|_{2,\mu}$. By   Fubini's
theorem, next Lemma \ref{a4}-(d)  applied with $g=F_{J,f}/\|F_{J,f}\|_{2,\P}$ and $T= A/\|F_{J,f}\|_{2,\P}$, it follows that
\begin{eqnarray*}
 \E\  \|F^{A,J}\|_{2,\mu}^2&=&
 \int_X
 \E\  |F_{J,f}(x)|^2\cdot {\bf 1}_{(|F_{J,f}(x)|\ge A)}\dd \mu(x)\cr 
 &\le &    6\int_X
\|F_{J,f}(x)\|_{2,\P}^2\
\exp\Big\{-{A^2\over 4\|F_{J,f}(x)\|_{2,{\P}}^2}\Big\}
\dd \mu(x) .
\end{eqnarray*}
 We have $\|F_{J,f}(x)\|_{2,\P}^2 \buildrel{{\rm a.e.}}\over{=}{1\over J} \sum_{j\le J}T_j(f^2)(x)$. By assumption (C), ${1\over J}\sum_{j\le
J} T_jf^2 $ converges to $1$ in $L^1(\mu)$, along some subsequence extracted from   $\mathcal J$, we can make this convergence almost everywhere too.
The requirement that $f\in L^\infty(\m)$, together with the  dominated convergence  theorem, then implies that
$$
 \int_X
\|F_{J,f}(x)\|_{2,\P}^2\
\exp\Big\{-{A^2\over 4\|F_{J,f}(x)\|_{2,{\bf P}}^2}\Big\}
\dd \mu(x) \to \exp\{- A^2 / 4\},
$$
along this index. 

    Extracting again if necessary  we obtain that 
 $
    \E \|F^{A,J}\|_{2,\mu}^2\le
 2 \exp\{- A^2 / 4\}, 
  $
along some index, which we still denote by $\mathcal{ J}$. Choose now $A= \sqrt{ 8\log M}$. We get 
 \begin{eqnarray}\label{estFA}
 \E \int\! \sup_{n\in I}|S_n(F^{A,J,\omega})|\ d\mu    \le   9\, \sqrt{  M} \,\exp\{- A^2 / 8\}\le      9\, M^{-1/2}  .  \end{eqnarray}
   Assume  that 
  \begin{equation}\label{delta} \min_{n,m\in I\atop n\neq m} \|S_n(f)- S_m(f) \|_{2,\mu}\ge \d   .\end{equation}
    Using Lemma \ref{a4}-(b), we get  
 \begin{eqnarray*}
 \label{debut1} 
\m \big\{x: \sup_{n\in I}|S_n(F_{A,J,\omega})(x)|\ge \frac{\gamma B\d}{2} \sqrt{\log M}- 9\, M^{-1/2}\big\}\ge   c /3  , \qq \forall \o\in D_J,
\end{eqnarray*}
  for all $J\in \mathcal{ J}$. 
  Let 
  $$\phi_{I,J,\omega}= \frac{F_{A,J,\omega}}{A}.$$ 
It follows that 
\begin{eqnarray}\label{debut2} 
\m \big\{x: \sup_{n\in I}|S_n(\phi_{I,J,\omega})(x)|\ge c'\d  \big\}\ge   c /3  , \qq \forall \o\in D_J,\end{eqnarray}
where $c'$ is a positive universal constant.
  Suppose that  for some $\d>0$,  $C(\delta) =\infty$.
 This means that we can select sets $I$ verifying \eqref{delta} with cardinality $M$ as large as we wish.  But 
 \begin{eqnarray}d(S^*(\phi_{I,J,\omega}),0)&\ge &d(\sup_{n\in I}|S_n(\phi_{I,J,\omega}),0)\cr &\ge & \int _{\sup_{n\in I}|S_n(\phi_{I,J,\omega})\ge  c'\d}\frac{\sup_{n\in I}|S_n(\phi_{I,J,\omega})|}{1+\sup_{n\in I}|S_n(\phi_{I,J,\omega})|}\dd \m
 \cr &\ge & (c/3)\frac{c'\d}{1+c'\d}  .
  \end{eqnarray}
 And we have 
 $$\E \|\phi_{I,J,\omega}\|_{2, \m}^2
  \le \frac{1}{ { 8\log M}}\E \int |F_{j,f}|^2 \dd \m\le \frac{1}{ { 8\log M}}.$$
Hence on a subset $D'_J$ of $D_J$ of positive measure, we have 
 $$ \|\phi_{I,J,\omega}\|_{\infty, \m}\le 1, \qq  \|\phi_{I,J,\omega}\|_{2, \m}\le K/\sqrt{  \log M}.$$
And $K$ depend on $c$ only.  Picking $\o$ in $D'_J$, $J$ varying, we deduce that $S^*$ cannot be continuous at $0$. 
  Hence a contradiction with \eqref{0cont}. This achieves the proof.   
 %%%%%%%%%%%%%%%%%%%%%%%%%%%%%%%%%%%%%%%%%%%%%%%%%%%%%%% 
\subsection{Proof of Theorem \ref{t7}} We start as in the proof of Theorem \ref{t2}. By using exactly the same arguments for proving  \eqref{fa}, 
we get here  \begin{eqnarray*}
%\label{fa1}
\E \int\! \sup_{n\in I}|S_n(F^{A,J })|\ d\mu & \le &\E \Big(
\sum_{n\in I} \int\! |S_n(F^{A,J })|^2\ d\mu \Big)^{{1/ 2}}\! \cr&\le&  \Big(
\sum_{n\in I} \int\! \E |S_n(F^{A,J })|^2\ d\mu \Big)^{{1/ 2}}  
  \, \le \,
\sqrt{M} S_1(I)\, \E \|F^{A,J }\|_{2,\mu} .  
\end{eqnarray*} 
Next estimate \eqref{estFA} is modified as follows.
 Let  $\a>1$ be some fixed real.  By extracting   we obtain that 
 $
    \E \|F^{A,J}\|_{2,\mu}^2\le
 \a \exp\{- A^2 / 4\}, 
  $
along some index,   still denoted $\mathcal{ J}$. Thus with \eqref{fa},
   \begin{equation}\label{estFA1}
 \E \int\! \sup_{n\in I}|S_n(F^{A,J,\omega})|\ d\mu     \le \sqrt{M} S_1(I)\,\E \|F^{A,J }\|_{2,\mu}   \le   6\, 
\sqrt{\a M}S_1(I)\,\exp\{- A^2 / 8\} .  
\end{equation}
  
    Let $ \delta=\min\big\{(\a -1)e^{-A^2/4\a},1)\big\}$
and $
\delta_k =\delta2^{-k}$, $k\ge 1$. We can   extract from   $\mathcal J$  a
subsequence $\mathcal J^*=\{J_k, k\ge 1\}$ depending  on    $f $ and $\a$, such that
$$
\mu\Big\{ \Big| {1\over J_k}\sum_{j\le J_k} T_jf^2 -1\Big| > \delta_k
\Big\} \le \delta_k, \qq \hbox{for all $k\ge 1$.}
$$
 Put   
$$
 B  = \Big\{ \forall
k\ge 1, \ \Big| {1\over J_k}\sum_{j\le J_k} T_jf^2 -1\Big| \le
\delta_k \Big\}.
$$
Plainly, 
 \begin{equation}\label{split}
\E\int \sup_{n\in I}|S_n(F_{A,J })|\
d\mu\le \E\int_{B } \sup_{n\in I}|S_n(F_{A,J })|\ d\mu + \E\int_{B^c}\sup_{n\in I}|S_n(F_{A,J })|\ d\mu.
\end{equation}
  The
first integral in the right-hand side of \eqref{split} can be bounded   for any $R>0$ by
 \begin{equation}\label{splitR}
\int_{B } \E\big(\sup_{n\in I}|S_n(F_{A,J })|\ {\bf 1}_{\{\|F_{A,J}\|_{2,\mu}>R\}}\big)\ d\mu + \int_{B }
\E\big(\sup_{n\in I}|S_n(F_{A,J })|\ {\bf 1}_{\{\|F_{A,J}\|_{2,\mu}\le R\}}\big)\ d\mu .
\end{equation}
Consider the first integral in \eqref{splitR}.
 The fact that $S_n$ is continuous on $L^\infty
(\mu)$ and Chebyshev's inequality allow  to write
 \begin{eqnarray*}\int_{B } \E\Big(\sup_{n\in I}|S_n(F_{A,J })|\ {\bf 1}_{\{\|F_{A,J}\|_{2,\mu}>R\}}\Big)\dd \mu &\le &   \E\Big(\big\|\sup_{n\in
I}|S_n(F_{A,J })| \big\|_{\infty, \m} \cdot {\bf 1}_{\{\|F_{A,J}\|_{2,\mu}>R\}}\Big)\cr &\le &
 AS_2(I)\,  \P\big\{  \|F_{A,J}\|_{2,\mu}>R\big\} \cr 
&\le &AS_2(I)\,   e^{- {R^2}/{4\a} }\E  \exp\Big\{ \frac{1}{4\a}\|F_{A,J }\|_{2,\mu}^2\Big\}.
\end{eqnarray*}
    
 We claim that for any $J\in \mathcal J^*$,
\begin{equation} \label{estexp}
\E  \exp\Big\{ \frac{1}{4\a}\|F_{A,J }\|_{2,\mu}^2\Big\}  \, \le \,
  \sqrt 2 +\a -1.
\end{equation}
 Admit this for a while. We get \begin{eqnarray}\label{firstsplitR}\int_{B } \E\Big(\sup_{n\in I}|S_n(F_{A,J })|\ {\bf 1}_{\{\|F_{A,J}\|_{2,\mu}>R\}}\Big)\dd \mu &\le &  AS_2(I)\,   e^{- {R^2}/{4\a} }  (
\sqrt 2 + \a -1).
\end{eqnarray} 
     Now we prove \eqref{estexp}. Let   $a=\frac{1}{4\a}$. At first by using   Jensen's inequality, 
\begin{eqnarray*}
 \E  \exp\big\{ a\|F_{A,J }\|_{2,\mu}^2\big\}&=&
\E  \exp\Big\{ a\int_X F_{A,J }^2\ d\mu\Big\}
 \le \E \int_X \exp \{ a
F_{A,J }^2 \}\ d\mu 
\cr &\le & \E \int_{B } \exp\big\{ a F_{A,J
}^2\big\}\ d\m + e^{aA^2}\m(B_\a^c).
\end{eqnarray*}

   Next on $B $, we have ${1\over J }\sum_{j\le J } T_jf^2\le 1+\d<\a$, so that 
$$
1-2a \Big( {1\over J }\sum_{j\le J } T_jf^2\Big) >
1-2a\a ={1\over 2} \qq \hbox{for all $J\in \mathcal J^*$.}$$  
As  $\E  e^{ bg^2} ={1\over \sqrt{1-2b}}$ if  $0\le b<{1\over 2}$,  we get
\begin{align*}
\int_{B } \E  \exp\big\{ a F_{A,J
}^2\big\}\ d\m  \le   \int_{B }  \E  \exp\big\{ a F_{
J}^2\big\}\ d\m
 =\int_{B }{\ d\mu \over \sqrt{1-2a \big( {1\over
J}\sum_{j\le J} T_jf^2\big)}} \le \sqrt 2.
\end{align*}
Hence for any $J\in \mathcal J^*$,
$$\E  \exp\big\{ a\|F_{A,J }\|_{2,\mu}^2\big\} \, \le \, \sqrt 2+ e^{aA^2}\m(B ^c)\, \le \,
\sqrt 2 + \delta e^{A^2a}\, \le \,  \sqrt 2 +\a -1.
$$

For the second integral in \eqref{splitR},
we have the  straightforward  bound
\begin{eqnarray}\label{secondsplitR}
\int_{B } \E\big(\sup_{n\in I}|S_n(F_{A,J })|\ {\bf 1}_{\{\|F_{A,J}\|_{2,\mu}\le R\}}\big)\ d\mu &\le &A
\sup_{\|h\|_{\infty,\mu}\le 1\atop \|h\|_{2,\mu}\le R/A}\int_{X }  \sup_{n\in I}|S_n(h)| \ d\mu.
\end{eqnarray}
By substituting   estimates  \eqref{firstsplitR}, \eqref{secondsplitR} into \eqref{splitR}, we can bound the first integral in the right-term of \eqref{split} as follows,
 \begin{equation}\label{splitRcomb}
 \E\int_{B } \sup_{n\in I}|S_n(F_{A,J })|\ d\mu \le AS_2(I)\,   e^{- {R^2}/{4\a} }  (
\sqrt 2 + \a -1)+ A
\sup_{\|h\|_{\infty,\mu}\le 1\atop \|h\|_{2,\mu}\le R/A}\int_{X }  \sup_{n\in I}|S_n(h)| \ d\mu.
\end{equation}
Consider the second integral in the right-term of   \eqref{split}. We use  Cauchy-Schwarz's inequality and the facts that $\m(B^c)\le \d$, $\E \|F_{A,J}\|_{2,\mu}\le \E \|F_{J}\|_{2,\mu}\le 1$, to get
\begin{eqnarray}\label{secondsplit}
  \E\int_{B^c}\sup_{n\in I}|S_n(F_{A,J })|\ d\mu &\le & \sqrt{\mu (B ^c) }\  \E  \big\|\sup_{n\in
I} |S_n(F_{A,J })|\big\|_{2,\mu}   \cr
 &\le & \sqrt \d\,\sqrt{M}\,S_1(I)\, \E \|F_{A,J}\|_{2,\mu}
\ \le\ 
\sqrt{\a -1 }\ e^{-A^2/8\a}\sqrt{M}\,S_1(I).
\end{eqnarray}      
  
By inserting  estimates  \eqref{splitRcomb}, \eqref{secondsplit} into \eqref{split},  we  next arrive to 
  \begin{eqnarray}\label{splitcomb}
\E\int \sup_{n\in I}|S_n(F_{A,J })|\
d\mu&\le &AS_2(I)\,   e^{- {R^2}/{4\a} }  (
\sqrt 2 + \a -1)\cr & &+ A
\sup_{\|h\|_{\infty,\mu}\le 1\atop \|h\|_{2,\mu}\le R/A}\int_{X }  \sup_{n\in I}|S_n(h)| \ d\mu  +\sqrt{\a -1 }\ e^{-A^2/8\a}\sqrt{M}\,S_1(I).
\end{eqnarray}
Now we insert  \eqref{estFA1}, \eqref{splitcomb}
into  \eqref{debut}, and next use estimate \eqref{estFA}. Picking $J$ arbitrarily in $\mathcal J^*$, we get 
  \begin{eqnarray}
\gamma \E  \sup_{n\in I} Z(S_n(f)) &\le& 6\, 
\sqrt{\a M}S_1(I)\,\exp\{- A^2 / 8\}+AS_2(I)\,   e^{- {R^2}/{4\a} }  (
\sqrt 2 + \a -1)\cr &  & + A
\sup_{\|h\|_{\infty,\mu}\le 1\atop \|h\|_{2,\mu}\le R/A}\int_{X }  \sup_{n\in I}|S_n(h)| \ d\mu  +\sqrt{\a -1 }\ e^{-A^2/8\a}\sqrt{M}\,S_1(I)\,.
\end{eqnarray}
  But  
$\a>1$ and $\g$ can be chosen arbitrarily   close to  $1$. We finally obtain, 
\begin{eqnarray}
  \E  \sup_{n\in I} Z(S_n(f)) &\le& 6\, 
\sqrt{  M}S_1(I)\,\exp\{- A^2 / 8\}+ A (
\sqrt 2  )S_2(I)\,   e^{- {R^2}/{4 } }    \cr &  & + A
\sup_{\|h\|_{\infty,\mu}\le 1\atop \|h\|_{2,\mu}\le R/A}\int_{X }  \sup_{n\in I}|S_n(h)| \ d\mu   .
\end{eqnarray}
     This last inequality being satisfied for any $f\in L^\infty (\mu)$ such that $\|f\|_{2,\mu}=1$,
 we easily deduce  the claimed
result   by continuity in  quadratic mean of $Z$.

        %%%%%%%%%%%%%%%%%%%%%%%%%%%%%%%%%%%%%%%%%%%%%%%%%%%%%%%%%%%%%%%%%%%%%%%%%%%%%%%%%%%%%%%%%%%%%%%%%%%%%%%%%%%%%%%%%%%%%%  
   \section{Kakutani--Rochlin's lemma}\label{s6}
 We conclude with this   extremely useful tool in ergodic theory. 
  \begin{lemma}\label{kr} If  $T$  is aperiodic, then for every  $\varepsilon > 0$
and for every  $n \ge 1$\index{Kakutani--Rochlin's lemma}  there exists
 $F \in \A$  such that the sets
 $F, T^{-1}(F) \dots T^{-(n-1)}(F)$  are mutually disjoint,
and such that we have, 
$$ {  \mu  \big(F \cup T^{-1}(F) \cup \dots \cup T^{-(n-1)}(F)
\big) > 1 - \varepsilon }.$$
\end{lemma}
 Any set
 $F \in \A$ satisfying the conclusions of Lemma \ref{kr} is  called
an $(\varepsilon,n)$-Kakutani--Rochlin set.
 \vskip 3 pt  We illustrate its usefulness by establishing two divergence criteria for    ergodic
summation methods.    The proof
is based on an argument due to     Deniel (see \cite{D}).  Let  
$ \big\{w_{n,k}, 1\le k
\le n, n\ge 1\big\}$  be a
triangular array of nonnegative  reals, and set  $
  W_{n
}=\sum_{k=1}^{n}w_{n,k}$, $n\ge 1$.  
 Consider an   automorphism $\tau$ from a probability space $(X,\mathcal A, \m)$. Put for $f\in L^0(\m)$,
$$T_n f(x) = {1\over W_n}\sum_{h=1}^{n}w_{n, h}   f(\tau^{h}x) .$$
\begin{theorem} \label{th1} Let $\varphi : \N \to \N$ be   such that $\lim_{n \to \infty}\varphi(n)= \infty$.  Assume that there
exist 
$  
\rho>0$, an infinite sequence
${\mathcal N}$ of integers  such that for any $n\in {\mathcal N}  $ 
\begin{equation} \label{wa}\min_{\varphi(n)\le j\le n-\varphi(n)}\Big({1\over W_{n  -j}}
  \sum_{k=  n-j-\varphi(n)}^{n-j-1}   w_{n-j,  k} \Big)
   \ge
 \rho ,  
\end{equation}
  and further that the series  $\sum_{n\in {\mathcal N}}  \varphi(n)/n  $ converges.
Let $0<\eta<\rho$. Then  there exists  $B\in {\mathcal   A}$
with $0<\m(B)\le \eta$ such that
$ \displaystyle\limsup_{{\mathcal N} \ni n\to \infty}T_n
\, \chi_B \ge \rho$   almost surely.
\end{theorem}

\begin{remark} \rm Suppose there exists a countable dense class ${\mathcal D}$ of functions from $L^1(\m)$ such that $\{ T_nf , n \in
{\mathcal N}\}$ converges almost everywhere to $\int fd\m$ for any  $f\in {\mathcal D}$. Then if condition (\ref{wa}) is satisfied,
there is no maximal inequality for the sequence $\{ T_n, n\in {\mathcal N}\}$. Indeed, otherwise by the Banach principle, we would have
that $\{ T_nf , n \in {\mathcal N}\}$ converges almost everywhere to $\int fd\m$ for any  $f\in L^1(\m) $. Taking $f=\chi_B$ where $B$ is
in the proposition above provides a contradiction.
\end{remark}
  Now let  
$
 \{w_{ k},   k\ge 1 \}$  be a sequence of non-negative  reals and  consider the ergodic sums
 $$A_n f(x) =  \sum_{h=1}^{n}w_{  h}   f(\tau^{h}x) .$$
 \begin{theorem} \label{th2} Let $\varphi : \N \to \N$ be   such that $\lim_{n \to \infty}\varphi(n)= \infty$.  Assume that there
exist 
$  
\rho>0$, an infinite sequence
${\mathcal N}$ of integers  such that for any $n\in {\mathcal N}  $ 
\begin{equation}   \D_n:= \min_{1\le h\le n-\varphi(n)}\Big( \sum_{k= 
h}^{h+\varphi(n)}w_{ k} 
\Big)     \to \infty,
\end{equation}
as $n\to \infty$ along ${\mathcal N}$,  and further that the series  $\sum_{n\in {\mathcal N}}  \varphi(n)/n  $ converges.
Let $0<\eta<\rho$. Then  there exists  $B\in {\mathcal A}$
with $0<\m(B)\le \eta$ such that
$ \displaystyle\limsup_{{\mathcal N} \ni n\to \infty}A_n
\, (\chi_B) =\infty$   almost surely.
\end{theorem} 

  %%%%%%%%%%%%%%%%%%%%%%%%%%%%%%%%%%%%%%%%%%%

\begin{proof}[Proof  of Theorem  \ref{th1}]  There is no loss of generality to assume
$\sum_{n\in {\mathcal N}}  \varphi(n)/n \le
\eta$.   By Rochlin's lemma,  for any $\varepsilon>0$, any integer $N $, there exists     $A\in {\mathcal A}$ such that $A,
TA,\cdots,T^{N-1}A$, are pairwise disjoint  and $ {1-\varepsilon}  \le N \mu(A)\leq 1
$. By applying it for $N=n$, $\varepsilon= {\varphi(n)}/{ n}$, we obtain that for each $n\in {\mathcal N} 
$,  there exists
$  A_n  \in {\mathcal A}$ such that
$A_n,
\tau^{}A_n,\ldots,
\tau^{n -1}A_n$ are mutually disjoint and $   \m\big(\sum_{u=0}^{n -1}\tau^{u}A_n \big)=n \m(A_n)\ge 1- {\varphi(n)}/{ n} $. 
Let
$$  B_n= \sum_{n -  \varphi(n) \le u<n } \tau^{ u}A_n, \qq \qq  D_n=\sum_{\varphi(n)\le j<n -  \varphi(n)}
\tau^{j}A_n.$$
  Then we have
\begin{eqnarray*}\m(B_n)&\le &  \varphi(n)  \m(A_n)\le    \frac{\varphi(n)}{ n},
\cr  \m(D_n )&\ge& \frac{n -  2\varphi(n)}{ n}(1-\frac{\varphi(n)}{ n})
\ge (1-2\frac{\varphi(n)}{ n})^2\ge 1-4\frac{\varphi(n)}{ n} .
\end{eqnarray*}
      Now let $0\le \ell<  n -  \varphi(n)$. As $\tau^\ell x\in B_n$ iff  $x\in \tau^{u-\ell}A_n$ for some $n -  \varphi(n)\le u<n $,  
we can
write 
    $$\chi_{B_n}(\tau^\ell x)= \sum_{n -  \varphi(n)\le u<n }\chi_{\{\tau^{u-\ell} A_n\}}(x)=\sum_{n -  \varphi(n)-\ell \le
 v<n -\ell }\chi_{\{\tau^{v}
  A_n\}}(x).
  $$
Let $\ell = n  -\varphi(n) -\l$ with $1\le \l <n -\varphi(n)$. 
We have  
 $$\chi_{B_n}(\tau^{n -\varphi(n) -\l} x)=  \sum_{\l \le v<\l +\varphi(n) }\chi_{\{\tau^{v} A_n\}}(x). $$
 As $\varphi(n)/n\to 0$ when  $  n\to \infty$ along ${\mathcal N}$, we have $2\varphi(n)\le n   $ once $n$ is large. Fix some
$\varphi(n)\le j<n -
\varphi(n) $ and  pick   $x\in 
\tau^{j}A_n$. If we   choose  $\l$  so that
$ 
\l\le j<
\l+\varphi(n) 
$,   by letting $v=j$ in  the equation above  we see that 
$\tau^{n  -\varphi(n) -\l} x\in B_n$. 
\smallskip\par Thus  $x\in 
\tau^{j}A_n$  and $\l\in\big\{ j-\varphi(n)+1, j-\varphi(n)+2, \ldots, j\big\}$ imply
$$  \tau^{n  -\varphi(n)
-\l} x\in B_n.$$
 Consequently, if $x\in 
\tau^{j}A_n$ 
 \begin{eqnarray}\label{basic}T_{n  -j} \chi_{{B_n}}(x)&=&  \sum_{k=1}^{n -j}w_{n-j,n-j-k}    \chi_{B_n}(\tau^{n-j-k}x)
 \ \ge \  \sum_{k=  1}^{\varphi(n)} w_{n-j,n-j-k}\chi
_{B_n}(\tau^{n  -j-k} x)\cr  
(k=\varphi(n)+\l-j)\qquad&= & \sum_{\l= j-\varphi(n)+1}^j w_{n-j,n -\varphi(n)-\l }\chi
_{B_n}(\tau^{n  -\varphi(n) -\l} x)
\cr &=& \sum_{\l= j-\varphi(n)+1}^j w_{n-j,n-\varphi(n)-\l }
\  =\  \sum_{k= 
1}^{\varphi(n)} w_{n-j,n-j-k} .
\end{eqnarray}
By the assumption made,
 \begin{eqnarray}\label{basicT} {1\over W_{n  -j}}\sum_{k= 
1}^{\varphi(n)} w_{n-j,n-j-k} &\ge & \min_{\varphi(n)\le j\le n-\varphi(n)}\Big({1\over W_{n  -j}}\sum_{k= 
1}^{\varphi(n)}w_{n-j,n-j-k} 
\Big)   
 \ \ge \  \rho  .
\end{eqnarray} 
  Note that $n -j>\varphi(n)$. Thus on $D_n $, 
 $$\sup_{m>\varphi(n)}\, T_{m} (\chi_{B_n} )   \ge  \rho. $$    
  Set 
$$E= \bigcup_{n\in {\mathcal N}}  B_n ,\qquad 
 F_N=\bigcap_{n\in {\mathcal N}\atop n\ge N}  D_n .$$  
 We observe that   $\m(F_N)\ge 1- 4\sum_{n\in {\mathcal N}\atop n\ge N}  \varphi(n)/n \to 1 $ as $N\to \infty$. Thus on $F_N$, 
\begin{equation}
\limsup_{{\mathcal N} \ni n\to\infty}\ T_n(\chi_E) \ge \rho.
\end{equation} Further $\m(E)\le \sum_{n\in {\mathcal N}}  \varphi(n)/n< \eta$. 
This establishes Theorem \ref{th1}.  
\end{proof}
\begin{proof}[Proof  of Theorem  \ref{th2}] We start   with (\ref{basic}) which here becomes
\begin{eqnarray*}\label{basicA}A_{n  -j} \chi_{{B_n}}(x) & \ge&    \sum_{k= 
1}^{\varphi(n)} w_{n-j-k} ,
\end{eqnarray*}
 and next modify the previous proof as follows: \begin{eqnarray*}\label{basicT}  \sum_{k= 
1}^{\varphi(n)} w_{n-j-k} &\ge & \min_{\varphi(n)\le j\le n-\varphi(n)}\Big( \sum_{k= 
1}^{\varphi(n)}w_{n-j-k} 
\Big)   
 \ \ge \   \min_{1\le h\le n-\varphi(n)}\Big( \sum_{k= 
h}^{h+\varphi(n)}w_{ k} 
\Big)   = \D_n.
\end{eqnarray*}  
  Thus 
 $\sup_{m>\varphi(n)}A_{m} (\chi_{B_n})    \ge  \D_n $ on $D_n$.
   Therefore on $F_N$, 
 $$\limsup_{{\mathcal N} \ni n\to\infty} T_n\chi_E =\infty .$$
  Further $\m(E)\le \sum_{n\in {\mathcal N}}  \varphi(n)/n< \eta$.  
\end{proof}

%%%%%%%%%%%%%%%%%%%%%%%%%%
%%%%%%%%%%%%%%%%%%%%%%%%%

\section{Appendix: GB and GC sets in ergodic theory}

  Let   $S_n\colon L^2(\mu)\to L^2(\mu)$, $n\ge 1$ be continuous operators satisfying assumption  {\rm (C)}. Assume that
property
$(\mathcal{ B}_p)$ is satisfied for some $2\le p<\infty$. By  the first
entropy criterion (Theorem \ref{t3})      the sets $C_f$ are GB  sets,  for any $f\in L^2(\mu)$.  
\vskip 2 pt 
Consider the following problem.  Let 
$A\subset L^2(\mu)$ and set
$$
C(A) = \{ S_n(f),\, n\ge 1, f\in A \}.
$$
Assume that  $A$  is a  GB  set. Can we say that   $C(A)$ is again a
 GB  set?
  In    \cite{We5},  
we  showed that so is the case  if $S_n(f)$ are  positive operators. Apart from this restriction, this  result can be viewed as a natural extension of the
first entropy criterion, since it is stated under the same assumptions and contains it obviously.

\begin{theorem} \label{c(a)}
    Let $S_n\colon L^2(\mu)\to L^2(\mu)$, $n\ge 1$, $S_1= \textit{Identity}$,  be  continuous operators  satisfying assumption  {\rm (C)}.
  Assume
that property $(\mathcal{ B}_p)$ is satisfied for some $2\le
p<\infty$. Let $A\subset L^2(\mu)$. Then we have the equivalence, 
 $$ \hbox{  A is a GB set}  \  \Longleftrightarrow \  \hbox{  $C(A)$ is a GB set.} $$
Further, there exists a constant $C$ such that $$\E \,   \sup_{h\in C(A)}Z(h)
\le C\bigl\{\inf_{h\in A}\|h\|_{2,\mu}+\E \, \sup_{h\in
A}Z(h)\bigr\}. $$
\end{theorem}

\begin{remarks} 1.  Since $S_1$ is the identity
operator,   $C(A)$ is a GB set only if $A$  is.

\noi 2.
 Let  $\tau$ be some ergodic endomorphism of  $(X,\A,\mu)$.  
By applying the above theorem with the choices $p=2$,
$$
  S_n(f)= {1\over
n}\sum_{k=0}^{n-1}f\circ \tau^k, \qq T_jf= f\circ \tau^j
$$  $T_j= T^j$ where $T$ is defined by $Tf=f\circ
\tau$, and by using Birkhoff's theorem, we deduce that $C(A)$ is a GB
set if $A$ is a GB set. 

Note that here the sets $C_f$ are all GC sets. This follows from Talagrand's estimate 
(Remark \ref{r1}) and Dudley's metric entropy theorem
\cite{Du2}. 
%M. Lacey (private communication) has an approach to prove  that $C(A)$ is a GC set if $A$ is a GC set.
\end{remarks}
We give a significantly simpler proof than in \cite{We5}.
\begin{proof}[Proof of Theorem \ref{c(a)}]  By  the Banach principle, there exists a constant
 $0<K<\infty$ such that  for any  $h\in L^p(\mu)$,
$$
\mu\big\{ \sup_{n\ge 1}|S_n(h)|\ge K\|h\|_{p,\mu}\big\} \le  {1\over
4}.
$$
Let $A_0\subset L^p(\mu)$ be  a finite  set. In view of the positivity  
of the operators $S_n$,
\begin{eqnarray*}
 \mu\Big\{  \sup_{h\in A_0}\sup_{n\ge 1}|S_n(F_{J,h})|\ge K\big\|\sup_{h\in
A_0}|F_{J,h}|\,\big\|_{p,\mu}\Big\}
 & \le &\mu\Big\{
\sup_{n\ge 1}S_n(\sup_{h\in A_0}|F_{J,h}|)\ge K\big\|\ \sup_{h\in
A_0}|F_{J,h}|\, \big\|_{p,\mu}\Big\} \cr &\le & {1\over 4}.
\end{eqnarray*}
Hence by integrating   with respect to    $\P$  and  by 
Fubini's theorem,    $$ \int_X  \P\Big\{  \sup_{h\in A_0}\sup_{n\ge 1}|S_n(F_{J,h})|\ge
K\|  \sup_{h\in A_0}|F_{J,h}|\, \|_{p,\mu}\Big\}  d\mu \le  {1\over
4}.
$$
Letting $D\subset X$ defined  by
$$
D =\Big\{   \P\big\{  \sup_{h\in A_0}\sup_{n\ge 1}|S_n(F_{J,h})|\le K\|\sup_{h\in
A_0}|F_{J,h}|\,\|_{p,\mu}\big\} \ge {1\over 2}  \Big\},
$$
it follows that \begin{eqnarray*}
  {1\over 4}   &\ge& \int_{D^c } \P\big\{  \sup_{h\in A_0}\sup_{n\ge
1}|S_n(F_{J,h})|\ge K\|\sup_{h\in A_0}|F_{J,h} | \,\|_{p,\mu}\big\}  \
d\mu
\ \ge \  {1\over 2}\mu(D^c)  
  .
\end{eqnarray*}
Thus $\m(D)\ge 1/2$. 
 As  $$ \P\Big\{ \big\| \sup_{h\in A_0}|F_{J,h}|
\big\|_{p,\mu}\ge 4 \E  \, \big\|  \sup_{h\in A_0}|F_{J,h}|
\big\|_{p,\mu}\Big\} \le   {1\over 4},$$
 we get on $D$
 $$       \P\Big\{  \sup_{h\in A_0} \sup_{n\ge
1}|S_n(F_{J,h})|\le 4K\E\, \big\|  \sup_{h\in A_0}|F_{J,h}|\,\big\|_{p,\mu}\Big\} \ge {1\over 4}   .   $$
 This along with \eqref{fer}  implies that on $D$
  \begin{eqnarray}\label{eq72}
     \E \sup_{h\in A_0}\sup_{n\ge 1}|S_n(F_{J,h})|\le  64K\,\E
\, \big\|  \sup_{h\in A_0}|F_{J,h}|\, \big\|_{p,\mu}   . \end{eqnarray} 
 \vskip 2 pt
   Let $N $ be some positive integer and $\e>0$.    Let $A\subset L^2(\mu)$ be a finite  set   such      that
 \begin{eqnarray}\label{eq73}
\|S_k(f)-S_\ell(g)\|_{2,\mu}\not= 0,
 \qq \qq \|S_k(f)\|_{2,\mu}\not= 0,  
\end{eqnarray}
for any $f,g\in A$ and   $k,\ell$ in $[1,N]$ with  $k\not= \ell$.
  \vskip 2 pt
    Choose now $A_0=\big\{ f^\e,f\in A\big\}\subset L^\infty(\m)$ be such that
  $
\sup_{f\in A}\|f-f^\e\|_{2,\mu}\le \e$.  
 By  continuity,   if $\e$ is small enough  we also have that       \begin{eqnarray}\label{eq73}
\t:=\min\Big\{ \|S_k(f)-S_\ell(g)\|_{2,\mu} ,
 \   \|S_k(f)\|_{2,\mu}, \ \  1\le k\not= \ell\le N,  f,g\in A\Big\}>0  .
 \end{eqnarray}
    From the commutation assumption also follows that
$$
\big\|S_k(F_{J,f^\e})-S_\ell(F_{J,g^\e})\big\|_{2,\P}^2
%= \big\|F_{J,S_k(f^\e)-S_l(g^\e) }\big\|_{2,\P}^2
 = {1\over J}\sum_{j\le J}\big(
T_j[S_k(f^\e)-S_\ell(g^\e)]\big)^2= {1\over
J}\sum_{j\le J}T_j\Big(S_k(f^\e)-S_\ell(g^\e)\Big)^2.
$$
 Let $q^* =\inf\big\{ q\ge 1 : 2^{-q}\le
{\t^2/ 4}\big\}$. By a routine extraction argument, there exists partial
index  $\mathcal{ J}=\{ J_q,q\ge q^* \}
 $ such that
 if $$
 A_q=\Big \{ \sup_{f,g\in A,\ 1\le k,\ell\le N } \Big |
\big\|S_k(F_{J_q,f^\e})-S_\ell(F_{J_q,g^\e})\big\|_{2,\P}^2-
\big\|S_k(f^\e)-S_\ell(g^\e)\big\|_{2,\mu}^2\, \Big |\ge 2^{-q}\Big \}
, 
$$
then  
 $    \mu(A_q) \le  {\e\over 2^q}$,  for all for any $q\ge q^*$.
 Thus 
$$H:=\bigcap_{q\ge q^*} A^c_q \ge 1-\sum_{q\ge q^*}  \mu(A_q)\ge 1-\e.$$
And on $H$,
\begin{eqnarray}  \label{eq75}  {1\over
2}\big\|S_k(f^\e)-S_\ell(g^\e)\big\|_{2,\mu}\le
\big\|S_k(F_{J,f^\e})-S_\ell(F_{J,g^\e})\big\|_{2,\P}\le
2\big\|S_k(f^\e)-S_\ell(g^\e)\big\|_{2,\mu},
 \end{eqnarray}
for any   $J\in \mathcal{ J}^*$ and any $f,g\in
A,\ 1\le k,\ell\le N$. 
  We note that  $\mathcal{ J}$ depends on $\e$, $N$
and $A$. 
 \vskip 2 pt 
By   applying Slepian's  inequality on $H\cap D$  and using \eqref{eq72} we get,  \begin{eqnarray}\label{eq76}
   \E \, \sup_{h\in  \{ S_n(f^\e), 1\le
n\le N,f\in A \} }Z(h) &   \le  & 2\E \,\sup_{f\in
A_0}\sup_{1\le n\le N}S_n(F_{J,f^\e})
\cr &\le & \ 128K\E \,\big\|\sup_{h\in A_0}|F_{J,h}|\big\|_{p,\mu},\end{eqnarray}
 for any $J\in \mathcal{ J}^*$.
   \vskip 4 pt

 We now estimate   $\E \,  \|   \sup_{f^\e\in
A_0}|F_{J,f^\e}|   \|_{p,\mu}$.  By using Jensen's inequality and integrability properties of Gaussian vectors, we get
 $$
 \E \big\| \sup_{f^\e\in
A_0}|F_{J,f^\e}|\, \big\|_{p,\mu}\le \Big( \E  \int_X
\sup_{f^\e\in A_0}|F_{J,f^\e}|^p\ d\mu \Big)^{ {1/ p}} \le C_p\Big(
\int_X \big( \E  \sup_{f^\e\in A_0}|F_{J,f^\e}|\big)^p  d\mu\Big)^{ {1/
p}}. 
$$
   The triangle inequality and
the symmetry properties of Gaussian laws further imply
\begin{align*}
\E \,
\sup_{f^\e\in A_0}|F_{J,f^\e}| & \le \E |F_{J,f_0^\e}|+\E \,
\sup_{f^\e,g^\e\in A_0}|F_{J,g^\e-f^\e}|  =\E |F_{J,f_0^\e}|+2\E \,
\sup_{f^\e\in A_0}F_{J,f^\e},
\end{align*}
where $f^\e_0\in A_0$ is arbitrary. Integrating then this inequality
over $H$ with respect to $\mu$, then applying the Slepian comparison
lemma, imply
\begin{align*}\int_H \big(\E \,  \sup_{f^\e\in
A_0}|F_{J,f^\e}|\big)^p\ d\mu &
 \le \int_H \big( \E   |F_{J,f_0^\e}|+ 2\E \, \sup_{f^\e\in A_0}F_{J,f^\e}\big)^p \ d\mu\cr
 {} & \le \int_H \big( \E \, |F_{J,f_0^\e}|+ 4\E \, \sup_{f^\e\in
A_0}Z(f^\e)\big)^p \ d\mu.
\end{align*}
 Hence,
$$
\Big(\int_H \big(\E    \sup_{f^\e\in
A_0}|F_{J,f^\e}|\big)^p\ d\mu\Big)^{ {1/ p}} \le \big\|{\bf 1}_H\E
|F_{J,f_0^\e}|\big\|_{p,\mu}+ 4\E \, \sup_{f^\e\in A_0}Z(f^\e).  
$$
 But
 $$
 \big\|{\bf 1}_{H }\E \,|F_{J,f_0^\e}|\,\big\|_{p,\mu}^p=\int_H \Big( {2\over \pi}
{1\over J}\sum_{j\in J}T_j(f_0^\e)^2\Big)^{ {p/2}}\ d\mu,
$$
and
$$
\Big( {1\over J}\sum_{j\in J}T_j(f_0^\e)^2\Big)^{ {p/ 2}}
\to \Big(\int_X (f_0^\e)^2\ d\mu \Big)^{ {p/ 2}},
$$
as
$J$ tends to infinity along   $\mathcal{ J}$, uniformly in $x\in H$.
  This shows that
\begin{eqnarray}\label{eq78}
\limsup_{J\to
\infty \atop J\in \mathcal{ J}} \Big(\int_H \big(\E \,  \sup_{f^\e\in
A_0}|F_{J,f^\e}|\big)^p\ d\mu\Big)^{ {1/ p}} \le  ( {2/ \pi})^{ {1/
2}} \big(\|f^\e_0\|_{2,\mu}+\e \big)+4\E  \sup_{f^\e\in A_0}Z(f^\e).
 \end{eqnarray}
Now  by
Jensen's inequality,
\begin{align*}
 \Big(\int_{H^c} \big(\E \,  \sup_{f^\e\in
A_0}|F_{J,f^\e}|\big)^p\ d\mu\Big)^{{1/ p}}
  \le
B\sqrt {\log\,
  \sharp(A_0) } \Big(\int_{H^c}  \sup_{f^\e\in A_0}\Big( {1\over
J}\sum_{j\in J}T_j(f^\e)^2\Big)^{ {p/ 2}}\Big)^{ {1/ p}}.
\end{align*}
But
$$
\sup_{f^\e\in A_0} {1\over J}\sum_{j\in J}T_j(f^\e)^2\to \sup_{f^\e\in
A_0}\int_X  (f^\e)^2\ d\mu ,
$$
 $\mu$-almost surely as $J$ tends to infinity along $\mathcal{ J}$. As moreover $$
 \sup_{h\in
A_0} {1\over J}\sum_{j\in J}T_j(f^\e)^2\le  \sup_{f^\e\in
A_0}\|f^\e\|_{\infty,\mu}^2 ,
$$
 by applying   the dominated convergence 
theorem  we get, 
 \begin{equation}\label{eq79}
\begin{split}
 \limsup_{J\to \infty \atop J\in \mathcal{ J}}
\Big(\int_{H^c} \big(\E \,  \sup_{f^\e\in A_0}|F_{J,f^\e}|\big)^p\
d\mu\Big)^{ {1/ p}}
 &\!  \!\le
 B(\mu(H^c))^{ {1/ p}}\sqrt {\log \sharp(A_0) }\sup_{f^\e\in
A_0}\|f^\e\|_{2,\mu}
\\
  &  \! \!  \le B(2\e)^{ {1/ p}}\sqrt {\log\,
 \sharp(A) }\big(\e+\sup_{f^\e\in A_0}\|f^\e\|_{2,\mu}\big) .
\end{split}
\end{equation} 
By combining now  estimates \eqref{eq76}, \eqref{eq78} and
\eqref{eq79}, we arrive to 
 \begin{equation}\label{eq710}
\begin{split}
  \E \,    \sup_{h\in  \{ S_n(f^\e),  f^\e\in
A_0,1\le n\le N  \} }Z(h) \le 32K\Big(\|f^\e_0\|_{2,\mu}+\e  +&4\E  \sup_{h\in
A_0}Z(h )
\\
  {}  +
B\sqrt {\log \sharp(A) }\big(\e+\sup_{f^\e\in A_0}&\|f^\e\|_{2,\mu}\big)  (2\e)^{ {1/ p} }\Big). 
\end{split}
\end{equation}
 But $\e$
is arbitrary as well as  $f_0^\e$ in  $A_0$. 
  By letting $\e$ tend to $0$ and using  $L^2$-continuity    of Gaussian vectors, we therefore conclude that 
   \begin{equation}\label{eq712}
 \E \, \Big\{  \sup_{h\in
 \{ S_n(f), 1\le n\le N,f\in A \} }Z(h)\Big\} \ \le \ 128K\, \Bigl(\inf_{h\in
A}\|h\|_{2,\mu}+4\E \sup_{h\in A}Z(h)\Bigr). 
\end{equation}
 The proof
is   achieved by letting   $A$ increase to some countable
$L^2(\mu)$-dense subset of  $A$. \end{proof}

  %%%%%%%%%%%%%%%%%%%%%  
 %%%%%%%%%%%%%%%%%%%%%  
 %%%%%%%%%%%%%%%%%%%%%  

 {}

  \end{document}